\documentclass{amsart}

\usepackage{amssymb}
\usepackage{mathtools}
\usepackage[shortlabels]{enumitem}
\usepackage{url}
\setlist{leftmargin=*}

\newtheorem{thm}{Theorem}[section] 
\newtheorem{cor}[thm]{Corollary}
 
\newtheorem{prop}[thm]{Proposition}
\newtheorem{claim}[thm]{Claim} 
\newtheorem{fact}[thm]{Fact}
\theoremstyle{definition} 
\newtheorem{defn}[thm]{Definition}

\theoremstyle{remark} 
\newtheorem{ntn}[thm]{Notation}
\newtheorem{rem}[thm]{Remark}

\newtheorem{sample}[thm]{Example} \numberwithin{equation}{section}

\newcommand{\KK}{\mathbf{K}}
\newcommand{\RR}{\mathbb{R}}

\newcommand{\UU}{\mathbb{U}}
\newcommand{\Th}{\text{Th}}

\newcommand{\tp}{\text{tp}}
\newcommand{\Av}{\text{Av}}
\newcommand{\Aut}{Aut}
\gdef\CF{\mathcal F}
\gdef\CM{\mathcal{M}}
\gdef\UU{\mathbb{U}}

\gdef\QQ{\mathbb{Q}}
\gdef\ZZ{\mathbb{Z}}
\gdef\NN{\mathbb{N}}
\gdef\KK{\mathbb{K}}

\gdef\dd{\operatorname{def}}
\gdef\CP{\mathcal{P}}
\gdef\CM{\mathcal{M}}

\begin{document}

\title{Regularity lemma for distal structures}

\author{Artem Chernikov} \address{Department of Mathematics,
University of California Los Angeles,
Los Angeles, CA 90095-1555, 
USA} \email{chernikov@math.ucla.edu}

\author{Sergei Starchenko} \address{Department of Mathematics, University of Notre Dame, Notre Dame,
  IN 46556, USA} \email{Starchenko.1@nd.edu}


\date{\today}

\begin{abstract}
  It is known that families of graphs with a semialgebraic edge relation of bounded complexity
  satisfy much stronger regularity properties than arbitrary graphs, and that they can be decomposed
  into very homogeneous semialgebraic pieces up to a small error (e.g., see \cite{PachSolymosi,
    AlonEtAl, GromovEtAl, fox2015polynomial}). We show that similar results can be obtained for
  families of graphs with the edge relation uniformly definable in a structure satisfying a certain
  model theoretic property called distality, with respect to a large class of generically stable
  measures. Moreover, distality characterizes these strong regularity properties. This applies in
  particular to graphs definable in arbitrary o-minimal structures and in $p$-adics.
\end{abstract}

\maketitle

\section{Introduction}

In this paper by a graph we always mean \emph{an undirected} graph, i.e.\ a graph $G=(V,E)$ consists
of a set of vertices $V$ together with \emph{a symmetric} set of edges $E \subseteq V\times V$.

As usual we say that a subset $V_0\subseteq V$ is \emph{homogeneous} if either $(v,v')\in E$ for all
$v\neq v'\in V_0$ or $(v,v')\notin E$ for all $v\neq v'\in V_0$, i.e.\ the induced graph on $V_0$ is
either complete or empty (we ignore the diagonal).

A classical theorem of Erd\H{o}s-Szekeres~\cite{erdos1935combinatorial} states that every graph on
$n$ vertices contains a homogeneous subset of size at least $\frac{1}{2}\log{n}$ (all $\log$'s are
of base two), and this bound is tight up to a constant multiple.

Since the families of graphs with a forbidden induced subgraph have much stronger structural
properties than arbitrary graphs, they have much bigger homogeneous subsets.
\begin{thm}[Erd\H{o}s-Hajnal, \cite{erdos1989ramsey}] For any finite graph $H$ there is a constant
  $c=c(H)>0$ such that every $H$-free finite graph on $n$ vertices contains a homogeneous subset of
  size at least $\displaystyle e^{c\sqrt{\log n}}$.
\end{thm}

However the following conjecture is widely open (see e.g. \cite{alon2001ramsey,
  chudnovsky2014erdos}).

\medskip
\noindent{\bf Erd\H{o}s-Hajnal Conjecture.}
For every finite graph $H$ there is a constant $\delta=\delta(H)>0$ such that every $H$-free graph
on $n$ vertices has a homogeneous subset of size at least $n^\delta$.  \medskip

In the bi-partite case one has better bounds.  Let $G=(V,E)$ be a graph.  We say that a pair of
subsets $V_1,V_2\subseteq V$ is \emph{homogeneous} if either $V_1\times V_2 \subseteq E$ or
$(V_1\times V_2) \cap E=\emptyset$.

\begin{thm}[Erd\H{o}s, Hajnal and Pach \cite{erdos2000ramsey}] For any finite graph $H$ there is a
  constant $\delta=\delta(H)>0$ such that every $H$-free graph on $n$ vertices has a homogeneous
  pair $V_1,V_2$ with $|V_1|,|V_2| \geq n^\delta$.
\end{thm}

The following definition is taken from \cite{FoxPach}.
\begin{defn} Let $\mathcal{G}$ be a class of finite graphs
  \begin{enumerate}
  \item $\mathcal{G}$ has the \emph{Erd\H{o}s-Hajnal Property} if there is $\delta>0$ such that
    every $G=(V,G)\in \mathcal{G}$ has a homogeneous subset $V_0$ of size $|V_0| \geq |V|^\delta$
  \item $\mathcal{G}$ has the \emph{strong Erd\H{o}s-Hajnal Property} if there is $\delta>0$ such
    that every $G\in \mathcal{G}$ has a homogeneous pair $V_1,V_2$ with $|V_1|,|V_2|\geq \delta|V|$.
  \end{enumerate}
\end{defn}

\begin{rem}\label{rem: strong EH implies EH} Is is shown in \cite{AlonEtAl} that if a family of
  finite graphs $\mathcal{G}$ has the strong Erd\H{o}s-Hajnal property and is closed under taking
  induced subgraphs then it has the Erd\H{o}s-Hajnal property.
\end{rem}

In this paper we consider families of graphs whose edge relations are given by a fixed definable
relation in a first-order structure.

\begin{defn}Let $\CM$ be a first-order structure and $R\subseteq M^k\times M^k$ be a definable
  relation. Consider the family $\mathcal{G}_R$ of all finite graphs $V=(G,E)$ where
  $G\subseteq M^k$ is a finite subset and $E=(V\times V)\cap R$. We say that $R$ satisfies the
  \emph{(strong) Erd\H{o}s-Hajnal property} if the family $\mathcal{G}_R$ does.
\end{defn}
 
We extend this notion to the bi-partite case.
\begin{defn} Let $\CM$ be a first-order structure and $R\subseteq M^m\times M^n$ a definable
  relation.
  \begin{enumerate}
  \item A pair of subsets $A\subseteq M^m, B\subseteq M^n$ is called \emph{$R$-homogeneous} if
    either $A\times B \subseteq R$ or $(A\times B) \cap R =\emptyset$.

  \item We say that the relation $R$ satisfies the \emph{strong Erd\H{o}s-Hajnal property} if there
    is a constant $\delta=\delta(R)>0$ such that for any finite subsets $A\subseteq M^m$,
    $B\subseteq M^n$ there are $A_0\subseteq A$, $B_0\subseteq B$ with $|A_0|\geq \delta|A|$,
    $|B_0|\geq \delta |B|$, and the pair $A_0,B_0$ is $R$-homogeneous.
  \end{enumerate}
\end{defn}

Our motivation for this work comes from the following remarkable theorem by Alon et al.

\begin{thm}[{\cite[Theorem 1.1]{AlonEtAl}}]\label{thm:semialg-Ramsey} If
  $R \subseteq \RR^n\times \RR^m$ is a semialgebraic relation then $R$ has the strong
  Erd\H{o}s-Hajnal property.
\end{thm}
 
\begin{rem}
  \begin{enumerate}[(i)]
  \item Although it is not stated explicitly in \cite{AlonEtAl}, but can be easily derived from the
    proof, homogeneous pairs in the above theorem can be chosen to be relatively uniformly
    definable.
  \item The above theorem was generalized by Basu (see \cite{Basu2}) to (topologically closed)
    relations definable in arbitrary o-minimal expansions of real closed fields.
  \end{enumerate}
  
\end{rem}

Besides the Erd\H{o}s-Hajnal property for semialgebraic graphs, the above theorem has many other
applications including unit distance problems \cite{pach2006diameter}, improved bounds in higher
dimensional semialgebraic Ramsey theorem \cite{conlon2014ramsey}, \cite[Theorem 1.2]{AlonEtAl},
algorithmic property testing \cite{fox2015polynomial}, and can also be used to obtain a strong
Szemer\'edi-type regularity lemma for semialgebraic graphs \cite{GromovEtAl, fox2015polynomial} (see
also Section~\ref{sec: distal regularity}).

The aim of this article is to demonstrate that the above result by Alon et al.,  along with its key implications,   holds at a much
larger level of generality, namely for families of graphs whose edge relation is definable in a
structure satisfying a certain model theoretic property called \emph{distality} (see Section
\ref{sec: Distality}) and with respect to the class of the so-called \emph{generically stable}
measures (as opposed to just the counting ones, see Section~\ref{sec: gen stab meas}). In
particular, this applies to graphs definable in arbitrary o-minimal structures and in $p$-adics with
analytic expansions, with respect to the Lebesgue (respectively, Haar) measure on a compact interval
(respectively, compact ball).

The following is one of the key results of our paper (see Theorem~\ref{thm:delta-main}).
\begin{thm}\label{thm:intro-main} Let $\CM$ be a distal structure, and $R\subseteq M^n\times M^m$ a
  definable relation.  Then there is a constant $\delta=\delta(R)>0$ such that for any generically
  stable measures $\mu_1,\mu_2$ on $M^n$ and $M^m$ respectively there are definable sets
  $A\subseteq M^n$, $B\subseteq M^m$ with $\mu_1(A)\geq \delta$, $\mu_2(B)\geq \delta$, and the pair
  $A,B$ is $R$-homogeneous.
\end{thm}

\begin{rem}
  It is not hard to see that our Theorem~\ref{thm:intro-main} implies 
  Theorem~\ref{thm:semialg-Ramsey}, by taking $\CM$ to be the ordered field of real numbers, and considering
  measures concentrated on finite sets.  Thus distal structures provide a natural framework for a
  model theoretic approach to Ramsey-type results in geometric combinatorics.
\end{rem}

\begin{rem} It is demonstrated by Malliaris-Shelah in \cite{malliaris2014regularity} (see also
  \cite{EHStable} for an alternative proof) that if $\CM$ is a stable structure and
  $R\subseteq M^k\times M^k$ is a definable relation then the family of finite graphs
  $\mathcal{G}_R$ has the Erd\H{o}s-Hajnal property. However in general this family does not have
  the \emph{strong} Erd\H{o}s-Hajnal property (see Section~\ref{sec: Failure in ACFp}).
 
\end{rem}

\begin{rem}
Our proof of Theorem \ref{thm:intro-main} (and the density version in Corollary \ref{cor: density extra parameters}) gives explicit bounds on $\delta$ and the number of the parameters in the definitions of $A$ and $B$ in terms of the VC-density of the edge relation $R$. In particular, for $o$-minimal structures and for $p$-adics, we obtain a bound in terms of the number of variables involved in $R$, due to the corresponding bounds for VC-density from \cite{VCD1}. We were informed by Pierre Simon that after reading our paper he had found another proof of this result which is faster, but does not give bounds.
\end{rem}

A brief summary of the paper. In Section~\ref{sec: prelims} we introduce the context and the
notation: first-order structures and definable sets, distality, Keisler measures and generic
stability.  In Section~\ref{sec: basicRamsey} we prove a definable generalization of Theorem
\ref{thm:semialg-Ramsey} for bi-partite graphs. In Section~\ref{sec: density distal Ramsey} we
improve it to a density version, using which we obtain an analogue for hypergraphs and a version
allowing additional parameters in the definition of the edge relation.  This gives in particular a
lot of new families of graphs satisfying the strong Erd\H{o}s-Hajnal property (see Example~\ref{ex:
  non-distal structures satisfying *}). In Section~\ref{sec: distal regularity} we obtain a strong
regularity lemma for hypergraphs definable in distal structures, generalizing the result for
semialgebraic hypergraphs from \cite{GromovEtAl, fox2015polynomial}.

In Section~\ref{sec: equiv to distality} we consider the converse to our results from the previous
sections.  First, in Section~\ref{sec: Failure in ACFp} we demonstrate a very explicit failure of
the definable counterpart of Theorem~\ref{thm:semialg-Ramsey} in the theory of algebraically closed
fields of positive characteristic, even without requiring definability of the homogeneous
subsets. It follows in particular that every field interpretable in a distal structure is of
characteristic $0$. In Section~\ref{sec: equivalence to distality} we prove that distality of a
structure is in fact equivalent to the definable counterpart of Theorem~\ref{thm:semialg-Ramsey}.

Some further questions concerning incidence phenomena and higher-dimensional Ramsey theory in our
setting will be addressed in a future paper.

\subsection*{Acknowledgements}
We thank Pierre Simon for pointing out a mistake in an earlier version of our results. We thank Dave Marker, Kobi Peterzil, Caroline Terry, Roland Walker and the anonymous referee for their numerous comments that have helped to improve the presentation.

The work presented in this paper began during the program on Model Theory, Arithmetic Geometry and
Number Theory at MSRI, Berkeley, Spring 2014. We thank MSRI for its hospitality. The first author
was partially supported by ValCoMo (ANR-13-BS01-0006), by the Fondation Sciences Mathematiques de
Paris (FSMP) and by a public grant overseen by the French National Research Agency (ANR) as part of
the Investissements d'Avenir program (reference: ANR-10-LABX-0098). The second author was partially
supported by NSF.

\section{Preliminaries}\label{sec: prelims}

\subsection{Model theoretic notation}
\label{sec:model-theor-triv}

We assume familiarity only with the very basic notions of model theory such as first-order
structures and formulas that can be found in any introductory model theory book
(e.g. \cite{marker2002model}). By a structure we always mean a first-order structure.

Our notations are standard.  We will denote first-order structures by script letters
$\mathcal{M}, \mathcal{N}$, etc., and use letter $M, N$ etc. to denote their underlying sets.  Very
often we will not distinguish singletons and tuples: e.g.\  we may use $x$ to denote a tuple of
variables $(x_1,\dotsc,x_n)$, use $a$ to denote an element of $M^n$, and then we use $|x|$ to denote
the length of the tuple $x=(x_1,\dotsc,x_n)$.

If $\CM$ is a structure and $\phi(x,y)$ is a formula in the language of $\CM$, then for
$a\in M^{|y|}$, as usual, by $\phi(M,a)$ we will denote the subset of $M^{|x|}$ defined by
$\phi(x,a)$, namely $\phi(M,a)=\{ b\in M^{|x|}\colon \CM\models \phi(b,a)\}$.

A subset $X\subseteq M^n$ is called \emph{definable} if there is a formula $\phi(x,y)$ and
$a\in M^{|y|}$ such that $X=\phi(M,a)$; if we want to specify the set of parameters, then for
$A\subseteq M$ a definable subset $X\subseteq M^n$ is called $A$-definable (or definable over $A$)
if we can choose $a$ as above in $A^{|y|}$.  Also if we want to specify $\phi$ we say that such a
set $X$ is $\phi$-definable.

\subsection{VC-dimension and NIP}\label{sec:vc-dimension-nip}

Vapnik--Chervonenkis dimension, or VC-dimension, is an important notion in combinatorics and
statistical learning theory.  Let $X$ be a set, finite or infinite, and let $\mathcal{F}$ be a
family of subsets of $X$.  Given $A \subseteq X$, we say that it is \emph{shattered} by
$\mathcal{F}$ if for every $A' \subseteq A$ there is some $S \in \mathcal{F}$ such that
$A \cap S = A'$. A family $\mathcal{F}$ is said be a \emph{VC-class} if there is some $n < \omega$
such that no subset of $X$ of size $n$ is shattered by $\mathcal{F}$. In this case \emph{the
  VC-dimension of $\mathcal{F}$}, that we will denote by $VC(\mathcal{F})$, is the smallest integer
$n$ such that no subset of $X$ of size $n+1$ is shattered by $\CF$.  For a set $B\subseteq X$, let
$\mathcal{F}\cap B=\left\{ A\cap B:A\in\mathcal{F}\right\}$ and let
$\pi_{\mathcal{F}}\left(n\right)=\max\left\{ \left|\mathcal{F}\cap B\right|:B\subseteq
  X,\left|B\right|=n\right\} $.

\begin{fact}[Sauer-Shelah lemma]\label{fac: SauerShelah} If $VC(\mathcal{F}) \leq d$ then for
  $n\geq d$ we have
  $\pi_{\mathcal{F}}\left(n\right)\leq\sum_{i\leq d}{n \choose i}=O\left(n^{d}\right)$.
\end{fact}

If $S\subseteq X$ is a subset and $x_1,\ldots,x_n \in X$, we let
$\Av(x_1,\ldots,x_n;S) = \frac 1 n |\{i\leq n: x_i\in S\}|$ (we don't assume that the points
$x_1,\dotsc,x_n$ are distinct).

\begin{fact}[VC-theorem \cite{vapnik1971uniform}, see also {\cite[Section 4]{NIP2}} for a
  discussion]\label{VC-theorem} For any $k > 0$ and $\varepsilon > 0$ there is
  $n =O(k (\frac{1}{\varepsilon})^2 \log \frac{1}{\varepsilon})$ satisfying the following.  For any
  \emph{finite} probability space $(X, \mu)$ and a family $\mathcal{F}$ of subsets of $X$ of
  VC-dimension $\leq k$, there are some $x_1, \ldots, x_n \in X$ such that for any
  $S \in \mathcal{F}$ we have $|\mu(S) - \Av(x_1, \ldots, x_n; S)| \leq \varepsilon$.
\end{fact}

\medskip

An important class of NIP theories was introduced by Shelah in his work on the classification
program \cite{ShelahClassification}. It has attracted a lot of attention recently, both from the
point of view of pure model theory and due to some applications in algebra and geometry. We refer to
\cite{AdlerNIP, SimBook} for an introduction to the area.

As was observed early on in \cite{Laskowski}, the original definition of NIP is equivalent to the
following one (see \cite{VCD1} for a more detailed account).
\begin{defn} Let $T$ be a complete theory and $\phi(x,y)$ a formula in $T$, where $x,y$ are tuples
  of variables, possibly of different length. We say that \emph{the formula $\phi(x,y)$ is NIP} if
  there is a model $\CM$ of $T$ such that the family of sets $\{ \phi(M,a) : a \in M^{|y|} \}$ is a
  VC-class. In this case we define the VC-dimension of $\phi(x,y)$ to be the VC-dimension of this
  class. (It is easy to see that by elementarily equivalence the above does not depend on the model
  $\CM$ of $T$.)

  \emph{A theory $T$ is NIP} if all formulas in $T$ are NIP.
\end{defn}

Slightly abusing terminology we say that a structure $\CM$ is NIP if its complete theory $\Th(\CM)$
is NIP. Restated differently, a structure $\CM$ is \emph{an NIP structure} if for every formula
$\phi(x,y)$ the family of $\phi$-definable sets
$\mathcal{F}_\phi=\{\phi(M,a) : a \in M^{|y|} \}$ is a VC-class.

\medskip Given a set of formulas $\Delta(x,y)$ and a set $B \subseteq M^{|y|}$, we say that $\pi(x)$
is a \emph{$\Delta$-type over $B$} if
$\pi(x) \subseteq \bigcup_{\phi(x,y) \in \Delta, b \in B} \left\{ \phi(x,b), \neg \phi(x,b) \right\}
$
and there is some $\mathcal{N} \succeq \CM$ and some $a \in N^{|x|}$ satisfying simultaneously all
formulas from $\pi(x)$. By a \emph{complete} $\Delta$-type over $B$ we mean a maximal $\Delta$-type
over $B$. We will denote by $S_{\Delta}(B)$ the collection of all complete $\Delta$-types over
$B$. In view of the remarks above, the following is an immediate corollary of the Sauer-Shelah
lemma.
\begin{fact}\label{fac: PolyTypesNIP} A structure $\CM$ is NIP if and only if for any finite set of
  formulas $\Delta(x,y)$ there is some $d \in \mathbb{N}$ such that $|S_{\Delta}(B)| = O(|B|^d)$ for
  any finite $B \subseteq M^{|y|}$.
\end{fact}
\medskip

\subsection{Distality}\label{sec: Distality} The class of distal theories is defined and studied in
\cite{Distal}, with the aim to isolate the class of purely unstable NIP theories (as opposed to the
class of stable theories which are always NIP, see also \cite{SimBook}). The original definition is
in terms of some properties of indiscernible sequences, but the following explicit combinatorial
characterization of distality given in \cite{ExtDefII} can be used as an alternative definition.

\begin{fact}\label{StrongHonestDef} Let $T$ be a complete NIP theory and $\CM$ a model of $T$. The
  following are equivalent.
  \begin{enumerate}
  \item $T$ is distal (in the sense of the original definition, see Fact~\ref{fac: distality
      indiscernibles}).
  \item For every formula $\phi(x,y)$ there is a formula $\psi(x,y_1,\dotsc,y_n)$ with
    $|y_1|=\dotsb=|y_n|=|y|$ such that: for any finite $B \subseteq M^{|y|}$ with $|B| \geq 2$ and
    any $a \in M^{|x|}$, there are $b_1,\dotsc,b_n \in B$ such that
    $\CM \models \psi(a,b_1,\dotsc,b_n)$ and $\psi(x,b_1,\dotsc,b_n) \vdash \mathrm{tp}_\phi(a/B)$
    (i.e.\  for any $b \in B$ either $\phi(M,b)\supseteq \psi(M,b_1,\dotsc,b_n)$ or
    $\phi(M,b)\cap \psi(M,b_1,\dotsc,b_n)=\emptyset$).

  \end{enumerate}
\end{fact}

\begin{rem}\label{srem:distal-one}
  It is not hard to see that if $\CM$ satisfies Fact~\ref{StrongHonestDef}(2) for all formulas
  $\phi(x,y)$ with $|x|=1$ then it satisfies it for all formulas, i.e. $\CM$ is distal.  Besides,
  any $\CM$ satisfying \ref{StrongHonestDef}(2) is automatically NIP (easy to see using the
  equivalence from Fact~\ref{fac: PolyTypesNIP}), so the assumption that $\CM$ is NIP in 
  Fact~\ref{StrongHonestDef} is used to deduce (2) from (1).
\end{rem}

\begin{rem}\label{rem: UDTFS}
  An immediate corollary of Fact~\ref{StrongHonestDef}(2) is that in a distal structure, for any
  formula $\phi(x,y)$ there is a formula $\psi'(y,y_1, \ldots, y_n)$ such that for any finite
  $B \subseteq M^{|y|}$ with $|B| \geq 2$ and $a \in M^{|x|}$, there are some
  $b_1, \ldots, b_n \in B$ such that $\phi(a,B) = \psi'(B, b_1, \ldots, b_n)$. Namely, one can
  take $\psi'(y,y_1, \ldots, y_n) = \forall x (\psi(x,y_1, \ldots, y_n) \rightarrow
  \phi(x,y))$. In fact, this corollary characterizes NIP (see \cite{ExtDefII} for the details).
\end{rem}

We list some examples of distal structures (providing more details than we normally would, for the
sake of a non model-theorist reader).

\subsubsection{O-minimal structures}

A structure $\CM = \left(M,<,\ldots\right)$ is o-minimal if every definable subset of $M$ is a
finite union of singletons and intervals (with endpoints in $M\cup\{\pm\infty\}$). From this
assumption one obtains cell decomposition for definable subsets of $M^{n}$, for all $n$. Moreover, a
cell decomposition of a definable set is uniformly definable in terms of its definition (see
\cite{van1998tame} for a detailed treatment of o-minimality, or \cite[Section 3]{scanlonminimality}
and references there for a quick introduction). Examples of o-minimal structures include
$\bar{\mathbb{R}} = \left(\mathbb{R},+,\times\right)$,
$\mathbb{R}_{\exp} = \left(\mathbb{R},+,\times,e^{x}\right)$,
$\mathbb{R}_{\text{an}} = \left(R,+,\times,f\restriction_{\left[0,1\right]^k}\right)$ for $f$
ranging over all functions real-analytic on some neighborhood of $[0,1]^k$, or the combination of
both $\mathbb{R}_{\text{an},\exp }$. It is straightforward to verify that if $\CM$ is o-minimal then
it satisfies Fact~\ref{StrongHonestDef}(2) for all formulas $\phi(x,y)$ with $|x|=1$.

\begin{sample}\label{sample-real} The field of reals $\RR$.

  By Tarski's quantifier elimination, for each $n$ the definable subsets of $\RR^n$ are exactly the
  semialgebraic sets, namely finite Boolean combinations of sets defined by polynomial equations
  $p(x_1,\dotsc,x_n)=0$ and inequalities $q(x_1,\dotsc x_n)>0$ for
  $p(\bar x),q(\bar x) \in \RR[x_1,\dotsc,x_n]$.  It is o-minimal, and so distal.
\end{sample}

\subsubsection{Ordered dp-minimal structures}

More generally, it is proved in \cite{Distal} that every ordered dp-minimal structure is distal (see
Fact~\ref{fac: dp-min distal totally indisc}). Examples of ordered dp-minimal structures include
weakly o-minimal structures and quasi-o-minimal structures. An ordered structure $\CM$ is weakly
o-minimal (quasi-o-minimal) if in every elementary extension, every definable subset is a finite
union of convex subsets (respectively, a finite boolean combination of singletons, intervals and
$\emptyset$-definable sets \cite{belegradek2000quasi}).

\begin{sample}\label{sample:integers} $\ZZ$ as an ordered group.

  By Presburger's quantifier elimination, for each $n$ the definable subsets of $\ZZ^n$ are finite
  boolean combinations of sets of the following types:
$$S_{\bar a}^= =\{ (x_1,\dotsc,x_n)\in \RR^n \colon
a_1x_1+\dotsb+a_nx_n=a_0\},$$
$$S_{\bar{a}}^> =\{ (x_1,\dotsc,x_n)\in \RR^n \colon 
a_1x_1+\dotsb+a_nx_n > a_0\},$$
 $$S^k_{\bar a}=\{ (x_1,\dotsc,x_n)\in \RR^n \colon \exists y\in \ZZ \,
 ky= a_0+a_1x_1+\dotsb+a_nx_n\},$$
 for $\bar a=(a_0,a_1,\dotsc,a_n)\in \ZZ^{n+1}$ and $k\in \NN$.  This structure is quasi-o-minimal
 (see \cite[Example 2]{belegradek2000quasi}).
\end{sample}

\begin{sample}\label{ex: RCVF}The \em{valued} field $\KK=\bigcup_{n\in \NN}\RR((T^{1/n}))$ of Puiseux
  power series over $\RR$, in the language
  $L_{\mathrm{div}} = \{0,1, <, +, -, \times, v(x) \leq v(y) \}$.

  Using quantifier elimination from \cite{dickmann1987elimination}, for each $n$ the definable
  subsets of $\KK^n$ (in the language of valued fields) are finite boolean combinations of sets of
  the following types:
$$S_p^= =\{ (x_1,\dotsc,x_n)\in \KK^n \colon
p(x_1,\dotsc,x_n)=0\},$$
$$S_{p}^> =\{ (x_1,\dotsc,x_n)\in \KK^n \colon
p(x_1,\dotsc,x_n)> 0\},$$
$$S_{p,q}^v =\{ (x_1,\dotsc,x_n)\in \RR^n \colon v(p(x_1,\dotsc,x_n)) \geq v(q(x_1,\dotsc,x_n))\},$$
for $p, q\in \mathbb \KK[x_1,\dotsc,x_n]$ and $k\in \NN$.

This structure is a model of the complete theory RCVF of real closed fields equipped with a proper
convex valuation ring, and by \cite{dickmann1987elimination} it is weakly o-minimal.

\end{sample}

\subsubsection{P-minimal structures with definable Skolem functions}

\begin{sample}\label{sample:qp} 
  By a result of Macintyre \cite{macintyre1976definable} the field of $p$-adics $\mathbb{Q}_p$
  eliminates quantifiers in the language $L_p = \{0,1,+,\times, v(x) \leq v(y), P_n(x)\}$, where for
  $n \geq 2$ we have $P_n(x) \iff \exists y (x = y^n)$.  It follows that for each $n$ the definable
  subsets of $\QQ_p^n$ are finite boolean combinations of sets of the following three types:
$$S_p =\{ (x_1,\dotsc,x_n)\in \RR^n \colon p(x_1,\dotsc,x_p)=0 \},$$
$$S_{p,q}^v =\{ (x_1,\dotsc,x_n)\in \RR^n \colon v(p(x_1,\dotsc,x_n))
\geq v(q(x_1,\dotsc,x_n))\},$$
$$S^k_{p}=\{ (x_1,\dotsc,x_n)\in \RR^n \colon \exists y\in \QQ_p \, y^k= p(x_1,\dotsc,x_n),$$ for
$p,q\in \QQ_p[x_1,\dotsc,x_n]$ and $k\in \NN$.
\end{sample}

Similarly to the o-minimal case, there is a notion of minimality for expansions of
$\mathbb{Q}_p$. Namely, a structure $\CM$ in a language $L \supseteq L_p$ is $p$-minimal if in every
model of $\Th(M)$, every definable subset in one variable is quantifier-free definable just using
the language $L_p$ \cite{haskell1997version}. $P$-minimal structures with an additional assumption
of definability of Skolem functions satisfy an analogue of the $p$-adic cell decomposition of
Denef. A motivating example of a $p$-minimal theory with definable Skolem functions is the theory
$p\textrm{CF}_{\text{an}}$ of the field of $p$-adic numbers $\mathbb{Q}_p$ expanded by all
sub-analytic subsets of $\mathbb{Z}_p$ \cite{van1999one}.

By \cite[Corollary 7.8]{VCD1}, every $p$-minimal structure with definable Skolem functions is
dp-minimal. Since $(\mathbb{Q}_p, +, \times, 0,1)$ is distal \cite{Distal} (can be also verified
using Fact~\ref{fac: dp-min distal totally indisc}), it follows by Remark~\ref{rem: dp-min dist exp}
that any $p$-minimal theory with Skolem functions is distal.

\subsection{Keisler Measures}\label{sec: gen stab meas}

Let $\CM$ be a structure.

Recall that a \emph{Keisler measure} on $M^n$ is a finitely additive probability measure on the
Boolean algebra of all definable subsets of $M^n$, i.e.\  it is a function $\mu$ that assigns to every
definable $X\subseteq M^n$ a number $\mu(X)\in [0,1]$ with $\mu(\emptyset)=0$, $\mu(M^n)=1$ and
\[ \mu(X\cup Y)=\mu(X)+\mu(Y) -\mu(X\cap Y) \]
for all definable subsets $X,Y\subseteq M^n$. Given a formula $\phi(x)$ with parameters from $M$ and
a Keisler measure $\mu$ on $M^{|x|}$, we will write $\mu(\phi(x))$ to denote $\mu(\phi(M^{|x|}))$.

In this paper we will deal mostly with the so-called generically stable and smooth Keisler measures.

\emph{Generically stable measures on $M^n$} are defined as Keisler measures on $M^n$ admitting a (unique) global $M$-invariant extension which is both finitely satisfiable in $M$ and definable over $M$ (see Remark \ref{rem: definable measure def}). In the NIP case, according to the following fact from
\cite[Theorem 3.2]{NIP3} (see also \cite[Section 7.5]{SimBook}), they can also be defined in terms
of the structure $\mathcal M$ alone without mentioning global measures.

\begin{fact}\label{GenStable} Let $\mathcal M$ be an NIP structure and $\mu$ a Keisler measure on
  $M^k$. Then the following are equivalent.
  \begin{enumerate}
  \item The measure $\mu$ is generically stable.
  \item For every formula $\phi(x,y)$ with $|x|=k$ and $\varepsilon >0$ there are some
    $a_1, \ldots, a_m \in M^k$ such that
    $|\mu(\phi(x,b)) - \Av \left( a_1, \ldots, a_m; \phi(x,b) \right)| < \varepsilon$ for any
    $b \in M^{|y|}$.
  \end{enumerate}
\end{fact}

The VC-theorem implies that in NIP theories, for \emph{any} Keisler measure, a uniformly definable
family of sets admits an $\varepsilon$-approximation by types (see \cite[Section 4]{NIP2}). 
Fact~\ref{GenStable}(2) implies that with respect to a generically stable measure, there are
$\varepsilon$-approximations by \emph{elements} of a model rather than just by types over it. We
remark that the bound on the size of $\varepsilon$-approximations depends just on the VC-dimension
of the formula (so uniform over all generically stable measures).

\begin{prop}\label{UniformEpsilonNetGenStable} Let $\CM$ be an NIP structure. Then for any
  $k\in \omega$ and any $\varepsilon > 0$ there is some
  $n =O(k (\frac{2}{\varepsilon})^2 \log \frac{2}{ \varepsilon})$ such that: for any formula
  $\phi(x,y)$ of VC-dimension at most $k$ and any generically stable measure $\mu$ on $M^{|x|}$,
  there are some $a_1, \ldots, a_n \in M^{|x|}$ such that for any $b \in M^{|y|}$,
  $|\mu( \phi(x,b) ) - \Av(a_1, \ldots, a_n;\phi(x,b))| < \varepsilon$.
\end{prop}

\begin{proof}

  Let $n \in \omega$ be given for $k$ and $\varepsilon$ by Fact~\ref{VC-theorem}.  Let $\phi(x,y)$
  be a formula of VC-dimension at most $k$ and $\mu$ an arbitrary generically stable measure on
  $M^{|x|}$.  By Fact~\ref{GenStable}(2) there are some $a_1', \ldots, a_m' \in M^{|x|}$ such that
  for any $b\in M^{|y|}$ we have $|\mu(\phi(x,b)) - \nu(\phi(x,b))| < \varepsilon/2$, where
  $\nu( \phi(x,b) ) = \Av(a_1', \ldots, a_m'; \phi(x,b))$. Now applying Fact~\ref{VC-theorem} to
  $\nu$, we find some $a_1, \ldots, a_n \in M^{|x|}$ such that for all $b \in M$,
  $|\nu(\phi(x,b)) - \Av(a_1, \ldots, a_n; \phi(x,b))| < \varepsilon/2$. Then
  $$|\mu(\phi(x,b)) - \Av(a_1, \ldots, a_n; \phi(x,b))| < \varepsilon$$ for all $b \in M^{|y|}$, as
  wanted.
\end{proof}

\begin{rem}\label{VCforSetsOfFormulas} Encoding several formulas into one we can replace a single
  formula $\phi(x,y)$ in Proposition~\ref{UniformEpsilonNetGenStable} by a finite set of formulas
  $\Delta(x,y)$.
\end{rem}

Recall that a Keisler measure $\mu$ on $M^n$ is called \emph{smooth} if there is a unique global
Keisler measure extending it.  The following equivalence can be used to avoid a reference to global
measures in the definition of smoothness.

\begin{fact}[{\cite[Section 2]{NIP3}}]\label{non-uniform epsilon-net for smooth} A Keisler measure
  $\mu$ on $M^n$ is smooth if and only if the
  following holds.\\
  For any formula $\phi(x,y)$ with $|x|=n$ and $\varepsilon > 0$ there are some formulas
  $\theta_i^1(x), \theta_i^2(x)$ and $\psi_i(y)$ with parameters from $M$, for $i=1, \ldots, m$,
  such that:
  \begin{enumerate}
  \item the sets $\psi_i(M^{|y|})$ partition $M^{|y|}$,
  \item for all $i$ and $b \in M^{|y|}$, if $\CM\models \psi_i(b)$, then
$$\CM\models (\theta^1_i(x) \rightarrow \phi(x,b)) \,\&\, (\phi(x,b)\rightarrow \theta_i^2(x)),$$
\item for each $i$, $\mu(\theta^2_i(x)) - \mu(\theta^1_i(x)) < \varepsilon$.
\end{enumerate}

\end{fact}

Every smooth measure is generically stable, and there are generically stable measures which are not
smooth (though every Keisler measure in an NIP theory can be extended to a smooth one, but over a
larger set of parameters). However, we have the following characterization from \cite{Distal}.
\begin{fact}\label{fac: distal iff gen stable is smooth} Let $T$ be NIP. Then the following are
  equivalent:
  \begin{enumerate}
  \item $T$ is distal.
  \item For any model $\CM$ of $T$, any generically stable measure on $M^n$ is smooth.
  \end{enumerate}
\end{fact}

\begin{rem}\label{rem: localizing generic stability} 
  Let $\mu$ be a Keisler measure on $M^n$ and $A\subseteq M^n$ a definable subset with
  $\mu(A)>0$. Then we can localize $\mu$ to $A$ by defining $\mu_A(X)=\mu(A\cap X)/\mu(A)$.  Clearly
  $\mu_A$ is a Keisler measure on $M^n$ and $\mu_A$ is generically stable (smooth) provided $\mu$
  is.
\end{rem}

Let $\CM$ be a structure, $\mu_1$ a Keisler measure on $M^m$ and $\mu_2$ a Keisler measure on
$M^n$. A Keisler measure $\mu$ on $M^{m+n}$ is called \emph{a product measure of $\mu_1$ and
  $\mu_2$} if for any definable subsets $X\subseteq M^m$, $Y\subseteq M^n$ we have
$\mu(X\times Y)=\mu_1(X)\mu_2(Y)$.  We can extend this notion to finitely many Keisler measure
$\mu_i$ on $M^{|n_i|}$ in an obvious way.  A product Keisler measure always exists but in general is
not unique.  However, for smooth measures we have the following proposition that follows from
\cite[Corollary 2.5]{NIP3}.
\begin{prop}\label{prop: unique smooth amalgam} Let $\CM$ be a
  structure, $\mu_1$ a smooth Keisler measure on $M^m$ and $\mu_2$ a smooth Keisler measure on $M^n$.  Then there
  is a unique product measure of $\mu_1$ and $\mu_2$ and this measure is also smooth.
\end{prop}
In the case of smooth measures $\mu_1$ and $\mu_2$ we will denote their unique product measure as
$\mu=\mu_1\otimes \mu_2$.

\medskip

Let $x_1,\dotsc, x_n$ be pairwise disjoint tuples of variables, and $\mu$ a Keisler measure on
$M^{|x_1|}\times \dotsb\times M^{|x_n|}$. Then for each $i=1,\dotsc,n$, $\mu$ induces a Keisler
measure $\mu_i$ on $M^{|x_i|}$ by
\[ \mu_i(Y)=\mu(M^{|x_1|}\times \dotsb\times M^{|x_{i-1}|}\times Y \times M^{|x_{i+1}|}\times \dotsb
\times M^{|x_n|}), \]
and we will denote this $\mu_i$ by $\mu |_{x_i}$. It is easy to see that if $\mu$ is generically
stable (smooth) then every $\mu|_{x_i}$ is
also generically stable (smooth). \\
Also in this case we will call a Keisler measure $\mu$ on $M^{|x_1|}\times \dotsb\times M^{|x_n|}$
\emph{ a product measure} if $\mu$ is a product of $\mu|_{x_1},\dotsc, \mu|_{x_n}$.

\medskip Finally, we give some examples of smooth Keisler measures.

\begin{fact}\label{ExamplesOfSmoothMeasures}
  \begin{enumerate}
  \item Any Keisler measure concentrated on a finite set (as Fact~\ref{GenStable}(2) is clearly
    satisfied).
  \item Let $\lambda_n$ be the Lebesgue measure on the unite cube $[0,1]^n$ in $\RR^n$.  Let
    $\mathcal M$ be an o-minimal structure expanding the field of real numbers.  If
    $X \subseteq \RR^n$ is definable in $\mathcal M$, then, by o-minimal cell decomposition,
    $X\cap [0,1]^n$ is Lebesgue measurable, hence $\lambda_n$ induces a Keisler measure on $M^n$.
    This measure is smooth by \cite[Section 6]{NIP3}.
  \item Similarly to (2), for every prime $p$ a (normalized) Haar measure on a compact ball in
    $\mathbb{Q}_p$ induces a smooth Keisler measure on $\QQ_p^n$ (see \cite[Section 6]{NIP3}).
  \item Any definable, definably compact group $G$ in an o-minimal structure or over the $p$-adics
    admits a unique $G$-invariant generically stable measure \cite{NIP2, DefAmenableNIP}, which is
    then smooth by distality and Fact~\ref{fac: distal iff gen stable is smooth}.
  \end{enumerate}
\end{fact}

\section{Strong Erd\H{o}s-Hajnal for definable bi-partite graphs in distal theories}\label{sec:
  basicRamsey}

In this section we prove the key result of this paper.

\begin{thm}\label{thm:delta-main} Let $\CM$ be a model of a distal theory, and $R\subseteq M^n\times M^m$ a definable relation. 
  Then there is a constant $\delta=\delta(R)>0$ and a pair of formulas $\psi_1(x,z_1)$,
  $\psi_2(y,z_2)$ such that for any generically stable measures $\mu_1,\mu_2$ on $M^n$ and $M^m$
  respectively, there are $c_1,c_2$ from $\CM$ with $\mu_1( \psi_1(M,c_1))\geq \delta$,
  $\mu_2(\psi_2(M,c_2))\geq \delta$, and the pair $A=\psi_1(M,c_1)$, $B=\psi_2(M,c_2)$ is
  $R$-homogeneous.
\end{thm}

As in \cite{AlonEtAl} the above theorem will follow from an asymmetric version (see Theorem
\ref{BasicPartitionGraphs} below).

Let $\CM$ be a distal structure and we fix a formula $\phi(x,y)$.  Let $\psi(x,y_1, \ldots, y_l)$ be
as given for $\phi(x,y)$ by Fact~\ref{StrongHonestDef}.

For $\vec d=(d_1,\dotsc,d_l)$ we will denote by $C_{\vec d}$ the subset of $M^{|x|}$ defined by
$\psi(x,\vec d)$ and call it \emph{a chamber}.  If in addition $\vec d\in B^l$ then we say that
$C_{\vec d}$ is a $B$-definable chamber.

\begin{defn}
  For a chamber $C=C_{\vec d}$ and $b\in M^{|y|}$ we say that $\phi(x,b)$ \emph{crosses $C$} if
  both $C\cap \phi(M,b)$ and $C\cap \neg\phi(M,b)$ are nonempty.
\end{defn}

For a chamber $C$ and a set $B$ we will denote by $C^\#(B)$ the set of all $b\in B$ such that
$\phi(x,b)$ crosses $C$.  Note that $C^\#(M)$ is a definable set (by a formula depending just on
the formula $\psi$ defining $C$).

\begin{defn}
  For $B\subseteq M^{|y|}$, a chamber $C$ is called \emph{$B$-complete} if $C$ is $B$-definable and
  $C^\#(B)=\emptyset$.
\end{defn}

It follows from the choice of $\psi$ that for every finite $B\subseteq M^{|y|}$ and $a\in M^{|x|}$
there is a $B$-complete chamber $C$ with $a\in C$. In particular, for any finite $B$ the union of
all $B$-complete chambers covers $M^{|x|}$.

\begin{defn}[$1/r$-cutting]\label{def:1r-cutting}
  Adopting a definition from \cite{ma}, we define $1/r$-cutting as follows.

  Let $\nu$ be a Keisler measure on $ M^{|y|}$. For a positive $r\in \RR$ we say that a family of
  chambers $\CF$ is a $1/r$-cutting with respect to $\nu$ if $M^{|x|}$ is covered by
  $\{ C \colon C\in \CF\}$ and for every $C\in \CF$ we have $\nu(C^\#(M)) \leq \frac{1}{r}$.
\end{defn}

The following claim is an analogue of a cutting lemma from \cite{ma} (see also Exercise 10.3.4(b)
there).

\begin{claim}\label{CrossCuttingExists}There is a constant $K$ such that the following holds. For
  any positive $r$ and for any generically stable measure $\nu$ on $M^{|y|}$ there is a finite set
  $S\subseteq M$ such that the family of all $S$-complete chambers is a $1/r$-cutting with respect
  to $\nu$, and the size of $S$ is bounded by $Kr^2\log 2r$.
\end{claim}
\begin{proof}
  Consider the family of sets
  \[\mathcal C=\{C^\#(M) \colon C \text{ is an $M$-definable chamber} \}. \]
  It is a definable family, hence has a bounded VC-dimension by NIP.  Applying 
  Proposition~\ref{UniformEpsilonNetGenStable} with $\varepsilon=1/r$ we obtain a subset 
  $S\subseteq M$ of size
  at most $K r^2 \log 2r$, where $K$ is a constant that depends only on the VC-dimension of
  $\mathcal C$, such that for every $M$-definable chamber $C$ if $\nu(C^\#(M)) > 1/r$ then
  $S\cap C^\#(M)\neq \emptyset$.

  Since $C^\#(S)=\emptyset$ for any $S$-complete cell $C$, we are done.
\end{proof}

The following theorem is an analogue of a result in \cite[Section 6]{AlonEtAl}.

\begin{thm}\label{BasicPartitionGraphs} Let $\CM$ be a distal structure and let $R(x,y)$ be a
  definable relation. Then for any $\beta \in (0,\frac{1}{2})$ there are some $\alpha \in (0,1)$ and
  formulas $\psi_1(x,z_1), \psi_2(y,z_2)$ depending just on $R$ and $\beta$ such that:

  for any Keisler measure $\mu$ on $M^{|x|}$ and any generically stable measure $\nu$ on $M^{|y|}$,
  there are some $c_1 \in M^{|z_1|}, c_2 \in M^{|z_2|}$ with $\mu(\psi_1(x,c_1)) > \alpha$,
  $\nu( \psi_2(y,c_2) ) > \beta$ and the pair of sets $\psi_1(M,c_1), \psi_2(M,c_2)$ is
  $R$-homogeneous.
\end{thm}
\begin{proof}

  Let $\phi(x,y)$ be a formula defining $R$.  Let $r$ be a positive real number that we will
  determine later.

  By Claim~\ref{CrossCuttingExists}, let $S \subseteq M^{|y|}$ be a set of size at most
  $Kr^2\log 2r$ such that for every $S$-complete chamber $C$, $\nu(C^{\#}(M)) \leq \frac{1}{r}$. It
  is not hard to see that there is a constant $K_1$ and a number $l = l(\psi) \in \mathbb{N}$ such the number of $S$-definable chambers is at
  most $K_1 |S|^l$.  Thus the number of $S$-complete chambers is at most $K' r^{2l}\log^l 2r$, where
  $K'$ is a constant.
 
  As the set of $S$-complete chambers covers $M^{|x|}$, there is an $S$-complete chamber $C_0$ with
  $\mu(C_0) \geq \frac{1}{K'r^{2l}\log^l (2r)}$.

  For the set $D = M^{|y|} \setminus C_0^{\#}(M)$, we have $\nu(D) \geq (1-\frac{1}{r})$ and for
  every $d\in D$, $\phi(x,d)$ does not cross $C_0$.  In particular all $a\in C_0$ have the same
  $\phi$-type over $D$. Note that $D$ is a disjoint union of
  $D_1 =\{ d\in D : C_0 \subseteq \phi(M,d) \}$ and
  $D_2 = \{ d\in D : C_0\cap \phi(M,d)=\emptyset \}$, and both $(C_0,D_1)$ and $(C_0,D_2)$ are
  $R$-homogeneous. Thus either $\nu(D_1) \geq \frac{1}{2} - \frac{1}{2r}$ or
  $\nu(D_2) \geq \frac{1}{2} - \frac{1}{2r}$.  Let $\psi_1 := \psi$ be the formula such that an instance of
  it defines $C_0$, and let $\psi_2$ be the formula such that an instance of it defines either $D_1$
  or $D_2$, depending on which one has large measure. By assumption there are only finitely many
  choices for both depending on the original data. So given $\beta \in (0,\frac{1}{2})$ we can find
  $r$ with $\frac{1}{2} - \frac{1}{2r} = \beta$, and take any positive
  $\alpha < \dfrac{1}{K'r^{2l}\log^l (2r)}$. Then, encoding finitely many choices for
  $\psi_1, \psi_2$ into one formula we can conclude the theorem. \end{proof}

\begin{proof}[Proof of theorem~\ref{thm:delta-main}]
  We can take any $\beta\in(0,\frac{1}{2})$ and let $\alpha$ be as in 
  Theorem~\ref{BasicPartitionGraphs}. Now take $\delta=\min \{ \alpha,\beta \}$.
\end{proof}

\begin{rem}
  We will see in Corollary~\ref{cor: density extra parameters} that one can allow an extra parameter
  in $R$ without affecting the uniform choice of $\psi_1, \psi_2$.
\end{rem}

\section{Density version and a generalization to hypergraphs}\label{sec: density distal Ramsey}

First we prove that Theorem~\ref{thm:delta-main} can be strengthened to a density version. It seems
that this implication is folklore, as it is mentioned in \cite[Corollary 7.1 and the remark
afterwards]{AlonEtAl} without definability of the homogeneous subsets and stated in
\cite{GromovEtAl}. However, the proofs in both places are very sketchy, so we give a complete proof
verifying definability of homogeneous subsets, and in addition working with Keisler measures. Our
argument is an elaboration on the proof of Theorem 3.3 in \cite{PachSolymosi}.

\begin{prop}\label{prop: DensityForGraphs} 
  Let $\CM$ be a distal structure and $R(x,y)$ a definable relation.
  Given $\alpha >0$ there is $\varepsilon >0$ such that for any Keisler measure $\mu$ on $M^{|x|}$,
  any generically stable measure $\nu$ on $M^{|y|}$, and a product measure $\omega$ of $\mu$ and
  $\nu$, if $\omega( R(x,y) ) \geq \alpha$ then there are uniformly definable (in terms of $\alpha$
  and $R$ only) $A_0 \subseteq M^{|x|}$ and $B_0 \subseteq M^{|y|}$ with
  $\mu(A_0) \geq \varepsilon$, $\nu(B_0) \geq \varepsilon $, and $A_0 \times B_0 \subseteq R$.
\end{prop}
We fix a distal structure $\CM$ and a definable relation $R(x,y)$.  By Theorem~\ref{thm:delta-main}
we know that there is a constant $\delta>0$ and formulas $\psi_1(x,z_1), \psi_2(y,z_2)$ such that
for any measure $\mu$ on $M^{|x|}$ and any generically stable measure $\nu$ on $M^{|y|}$ there are
some $A \subseteq M^{|x|}$ and $B\subseteq M^{|y|}$ definable by an instance of $\psi_1$ and
$\psi_2$ respectively, with $ \mu(A) \geq \delta$ and $\nu(B) \geq \delta$, such that either
$A \times B \subseteq R$ or $A\times B \cap R =\emptyset$.

Now we fix Keisler measures $\mu,\nu$ as in the proposition and let $\omega$ be a product Keisler
measure of $\mu, \nu$ on $M^{|x| + |y|}$.

For definable sets $A \subseteq M^{|x|},B \subseteq M^{|y|}$ we denote by $d(A,B)$ the density of
$R$ in $A\times B$, namely
\[
d(A,B)=\frac{ \omega( (A\times B ) \cap R )}{\mu(A)\nu(B)},
\]
and setting $d(A,B)=0$ if $\mu(A)\nu(B)=0$.  The following claim is a basic step.
\begin{claim}\label{claim:basic-step}
  Let $A\subseteq M^{|x|}$, $B\subseteq M^{|y|}$ be definable sets with $d(A,B) > 1-\delta^2$. Then
  there are subsets $A_1\subseteq A$, $B_1\subseteq B$ defined uniformly (in terms of $A$, $B$ and
  $R$) such that $\mu(A_1)\geq \delta \mu(A)$, $\nu(B_1)\geq \delta \nu(B)$ and
  $A_1\times B_1 \subseteq R$.
\end{claim}
\begin{proof}
  Using Remark~\ref{rem: localizing generic stability} we apply Theorem~\ref{thm:delta-main} to
  $\mu_A, \nu_B$ --- the localizations of $\mu$ on $A$ and $\nu$ on $B$, respectively. This gives us
  $A'\subseteq M^{|x|}$, $B'\subseteq M^{|y|}$ defined by instances of $\psi_1$ and $\psi_2$
  respectively, such that for the sets $A_1=A'\cap A$, $B_1=B'\cap B$ we have
  $\mu(A_1)\geq \delta \mu(A)$, $\nu(B_1)\geq \delta \nu(B)$ and either $A_1\times B_1 \subseteq R$
  or $\left( A_1\times B_1 \right) \cap R=\emptyset$.

  If $\left( A_1\times B_1 \right) \cap R=\emptyset$ then
  \[\omega((A\times B)\cap R)
  \leq \omega(A\times B)-\omega(A_1\times B_1)\leq (1-\delta^2)\mu(A)\nu(B), \]
  contradicting the assumption $d(A,B) > 1-\delta^2$.
\end{proof}

It is not hard to see that Proposition~\ref{prop: DensityForGraphs} follows from Claim
\ref{claim:basic-step} and the following claim by iterating sufficiently (but boundedly) many times
and taking the conjunction of the corresponding defining formulas.

\begin{claim}\label{claim:ind}
  For any $0< \alpha < 1 - \delta^2$ there is some $h >0$ such that for any definable $A$ and $B$ with
  $d(A,B)\geq \alpha$ there are uniformly definable (in terms of $R$, $\alpha$, $A$, $B$) subsets
  $A' \subseteq A$ $B'\subseteq B$ with $\mu(A') \geq h \mu(A)$, $\nu(B') \geq h \nu(B)$ and
  $d( A',B')\geq \alpha\dfrac{1}{1-\delta^2}$.
\end{claim}
\begin{proof}
  We pick $d \in (0,1)$ to be determined later.

  We choose $A_0\subseteq A$, $B_0\subseteq B$ homogeneous with respect to $R$ and with
  $\mu(A_0) \geq \delta \mu(A)$ and $\nu(B_0)\geq \delta \nu(B)$ (applying 
Theorem~\ref{BasicPartitionGraphs} and Remark~\ref{rem: localizing generic stability} to $\mu_A, \nu_B$
  --- the localizations of $\mu$ on $A$ and $\nu$ on $B$, respectively).

  If $A_0\times B_0 \subseteq R$ then we take $A'=A_0$, $B' = B_0$, $h = \delta$ and we are done. So
  assume
  \begin{equation}
    \label{eq:1}
    (A_0\times B_0)\cap R =\emptyset.  
  \end{equation}

  Let $a_0=\dfrac{\mu(A_0)}{\mu(A)}$ and $b_0=\dfrac{\nu(B_0)}{\nu(B)}$.  Let $\alpha'=d(A,B)$, so
  $\alpha' \geq \alpha$.

  From \eqref{eq:1} it follows that $a_0b_0 \leq 1-\alpha' \leq 1-\alpha$.  In particular at least
  one of $a_0$ or $b_0$ is at most $\sqrt{1-\alpha}$.
  \\
  {\bf We assume $a_0 \leq \sqrt{1-\alpha}$.}

  \medskip

  Let $A_1=A\setminus A_0$ and $B_1 =B\setminus B_0$.

  \medskip

  \noindent{ \bf Case 1: $\nu(B_0) \leq  d \nu(B)$.} \\
  For definable $A'\subseteq A$ and $B'\subseteq B$ we write $\omega(R(A',B'))$ for
  $\omega( (A'\times B') \cap R)$.

  Since there are no $R$-edges between $A_0$ and $B_0$ we have
  \begin{equation}
    \label{eq:2}
    \omega(R(A_0,B_1)) +\omega(R(A_1,B_0))+\omega((A_1,B_1)) = \omega(R(A,B)) \geq \alpha \mu(A) \nu(B)
  \end{equation}
  We also have
  \begin{multline}
    \label{eq:3}
    \mu(A_0)\nu(B_1)+\mu(A_1)\nu(B_0)+\mu(A_1)\nu(B_1) \\= \mu(A) \nu(B)- \mu(A_0)\nu(B_0) \leq
    \mu(A)\nu(B)(1- \delta^2)
  \end{multline}

  A very simple combinatorial statement is that if $r_1+r_2+r_3 \geq r$ and $s_1+s_2+s_3 \leq s$
  then there is $i\in \{1,2,3 \}$ with $\dfrac{r_i}{s_i}\geq \dfrac{r}{s}$.

  So in this case we can choose $A'\in \{ A_0,A_1\}$ and $B'\in \{B_0,B_1\}$ such that
  \[ d(A',B')\geq \alpha \frac{1}{1-\delta^2},\]
  and $\mu(A')\geq h \mu(A)$, $\nu(B')\geq h\nu(B)$, where
  \[ h=\min\{ \delta, (1-d), 1-\sqrt{1-\alpha}\}. \]

  \medskip

  \noindent{\bf Case 2:  $\nu(B_0)  > d \nu(B)$.} \\
  In this case we will take $A' =A_1=A\setminus A_0$ and $B'=B$.  As above, let
  $B_1=B\setminus B_0$. Let $d'=1-d$.

  The maximal possible measure of the set of $R$-edges between $A_0$ and $B$ is
  \[ \omega(R(A_0,B))=\omega(R(A_0, B_1))\leq \mu(A_0)\nu(B_1) \leq
  \sqrt{1-\alpha}d'\mu(A)\nu(B). \]

  Thus for the measure of the set of $R$-edges between $A_1$ and $B$ we obtain
  \[ \omega(R(A_1,B))\geq \alpha \mu(A) \nu(B)- \sqrt{1-\alpha}d' \mu(A) \nu(B).\]

  Since $\mu(A_1)\nu(B) \leq (1-\delta) \mu(A) \nu(B)$ we obtain

  \[ d(A_1,B) =\frac{\omega(R(A_1,B))}{\mu(A_1)\nu(B) } \geq \frac{ \alpha -
    \sqrt{1-\alpha}d'}{1-\delta} = \alpha\frac{ 1-\sqrt{1-\alpha} (d'/\alpha)}{1-\delta}.\]

  As $d'$ decreases to $0^+$, the right side of the above inequality goes increasingly to
  $\alpha \dfrac{1}{1-\delta}$.  Since $0< \delta <1$ we have that
  $\dfrac{1}{1-\delta^2} < \dfrac{1}{1-\delta}$.  So we can choose $d\in(0,1)$ so that for $d'=1-d$
  the right side is at least $\alpha\dfrac{1}{1-\delta^2}$.

  \bigskip

  Combining the two cases together we take
  \[ h=\min\{ \delta, (1-d), 1-\sqrt{1-\alpha}\}. \]

  Uniform definability of $A',B'$ in all the cases follows from the uniform definability of
  $A_0,B_0$ and construction, so as always we can encode finitely many formulas into a single one.
\end{proof}

Now we can use this proposition inductively to prove the analogue of Proposition~\ref{prop:
  DensityForGraphs} for hypergraphs, essentially following the proof of \cite[Theorem
8.2]{GromovEtAl}.
\begin{prop}\label{prop: DensityRamseyHypergraphs}\label{HypergraphDensityVersion} 
  Let $\CM$ be a distal
  structure and $R(x_0, \ldots, x_{h-1})$ a definable relation. Given $\alpha >0$ there is
  $\varepsilon >0$ such that: given a generically stable product measure $\omega$ on
  $M^{|x_0|} \times M^{|x_1|} \times \cdots \times M^{|x_{h-1}|}$ with $\omega(R) \geq \alpha$ there
  are definable sets $A_i \subseteq M^{|x_i|}$ with $\omega|_{x_i}(A_i) \geq \varepsilon$ for all
  $i<h$ such that $\prod_{i<h} A_i \subseteq R$.  Moreover, each $A_i$ is defined by an instance of
  a formula that depends only on $R$ and $\alpha$.
\end{prop}
\begin{proof}
  Let $h \geq 2$ be given, and assume inductively that we have proved the proposition for all
  $i \leq h$. Let $R(x_0, \ldots, x_{h})$ and $\alpha>0$ be given. Let $\omega$ be a generically
  stable product measure on $M^{|x_0|}\times \cdots \times M^{|x_{h}|}$. Applying 
  Proposition~\ref{prop: DensityForGraphs} with $h=2$ to the \emph{binary} relation $R(x_0; x_1, \ldots, x_h)$
  we find some $\varepsilon' > 0$, $A_0$ with $\omega|_{x_0}(A_0) \geq \varepsilon' $ and
  $A \subseteq M^{|x_1|} \times \cdots \times M^{|x_{h}|}$ with
  $\omega|_{x_1,\ldots,x_h}(A) \geq \varepsilon'$ such that $R$ holds on all elements of
  $A_0 \times A$ (the corresponding projections of $\omega$ are clearly generically
  stable). Moreover, $A = R'(M^{|x_1|} \times \cdots \times M^{|x_h|})$ for some uniformly definable
  (depending only on $R$ and $\alpha$) relation $R'$.  We apply the inductive assumption to $R'$
  with $h-1$, $\alpha = \varepsilon'$ and $\omega|_{x_1, \ldots, x_h}$, which gives us some
  $\varepsilon'' > 0$ and uniformly definable sets $A_i, 1\leq i \leq h$ with
  $\omega|_{x_i}(A_i) \geq \varepsilon ''$ and such that
  $A_1 \times \cdots \times A_h \subseteq R'$, which implies
  $A_0 \times A_1 \cdots \times A_h \subseteq R$. Take
  $\varepsilon = \min\{ \varepsilon', \varepsilon'' \}$. All the data is chosen uniformly depending
  only on $R,\alpha$.
\end{proof}

The density version implies a generalization of Theorem~\ref{BasicPartitionGraphs} for hypergraphs.
\begin{cor}\label{DensityForGraphs}
  Let $\CM$ be a distal structure and $R(x_0, \ldots, x_{h-1})$ a definable relation. Then there is
  $\delta >0$ such that for any generically stable measures $\mu_i$ on $M^{|x_i|}$, there are $A_i$
  with $\mu_i(A_i) \geq \delta$ for all $i<h$, uniformly definable in terms of $R$, and such that
  either $\prod_{i<h} A_i \subseteq R$ or $\prod_{i<h} A_i \cap R =\emptyset$.
\end{cor}
\begin{proof}
  Since a product of generically stable measures is generically stable, the measure
  $\omega = \mu_0 \otimes \cdots \otimes \mu_{h-1}$ is generically stable, and either
  $\omega(R) \geq \frac{1}{2}$ or $\omega (\neg R) \geq \frac{1}{2}$. Applying 
  Proposition~\ref{HypergraphDensityVersion} with $\alpha = \frac{1}{2}$ to $R$ and to $\neg R$ we obtain some
  $\varepsilon_1, \varepsilon_2$ respectively. But then
  $\delta = \min\{\varepsilon_1, \varepsilon_2\}$ satisfies the conclusion.
\end{proof}
Besides, the formulas defining homogeneous subsets can be chosen depending just on the
formula defining the edge relation, and not on the parameters used (in the semialgebraic setting
this corresponds to saying that the complexity of the homogeneous subsets is bounded in terms of the
complexity of the edge relation, and does not depend on the choice of the coefficients of the
polynomials involved).

\begin{cor}\label{cor: density extra parameters} 
  Let $\CM$ be a distal structure and
  $\phi(x_0, \ldots, x_{h-1}, y)$ a formula.  Given $\alpha >0$ there is $\varepsilon >0$ such that:
  for a definable relation $R(x_0, \ldots, x_{h-1}) = \phi(x_0, \ldots, x_{h-1}, c)$ with some
  $c \in M^{|y|}$ and a generically stable product measure $\omega$ on
  $M^{|x_0|} \times M^{|x_1|} \times \cdots \times M^{|x_{h-1}|}$ with $\omega(R) \geq \alpha$ there
  are definable sets $A_i \subseteq M^{|x_i|}$ with $\omega|_{x_i}(A_i) \geq \varepsilon$ for all
  $i<h$ and $\prod_{i<h} A_i \subseteq R$.  Moreover, each $A_i$ is defined by an instance of a
  formula that depends only on $\phi$ and $\alpha$.
\end{cor}
\begin{proof}
  Follows immediately by Proposition~\ref{prop: DensityRamseyHypergraphs} applied to the
  relation $$R'(x_0, \ldots, x_{h-1},y) = \phi(x_0, \ldots, x_{h-1},y)$$ and to the generically
  stable product measure $\omega' = \omega \otimes \delta_c$, where $\delta_c$ is a (generically
  stable) $\{ 0,1 \}$-valued measure on $M^{|y|}$ concentrated on $c$.
\end{proof}

\begin{sample}
  Let $\lambda_n$ be the Lebesgue measure on $\mathbb{R}^n$ restricted to the unit cube,
  i.e. $\lambda_n(X) = \Lambda_n(X \cap I_n)$ where $\Lambda_n$ is the standard Lebesgue measure and
  $I_n$ is the unit cube in $\mathbb{R}^n$.

  Let $\mathcal{R}$ be an o-minimal expansion of $\mathbb{R}$ and $R(x_1, \ldots, x_n;u)$ be a
  formula. Then for any $\alpha > 0$ there is some $\varepsilon > 0$ such that for any
  $c \in \mathbb{R}^{|u|}$ with $\lambda_n(R(\mathbb{R}^n;c)) \geq \alpha$ there are definable
  $A_i \subseteq \mathbb{R}$, $i=1, \ldots, n$ with $\lambda_1(A_i) \geq \varepsilon$ and
  $A_1 \times \cdots \times A_n \subseteq R(\mathbb{R}^n,c)$.

  This follows from Corollary~\ref{cor: density extra parameters} and Fact~\ref{ExamplesOfSmoothMeasures}.
\end{sample}

Also we get a generalization of the original semialgebraic counting version over finite sets from
Theorem~\ref{thm:semialg-Ramsey} with additional control on the parameters over which the
homogeneous subsets are defined.

\begin{cor}\label{BasicRamseyForCountingMeasures} Let $\CM$ be a distal structure and let a formula
  $\phi(x,y,z)$ be given. Then there is some $\delta= \delta(\phi)>0$ and formulas $\psi_1(x,z_1)$
  and $\psi_2(y,z_2)$ depending just on $\phi$ and satisfying the following. For any definable
  relation $R(x,y) = \phi(x, y, c)$ for some $c \in M^{|z|}$ and finite
  $A \subseteq M^{|x|}, B \subseteq M^{|y|}$ there are some $A' \subseteq A, B' \subseteq B$ with
  $|A'| \geq \delta|A|, |B'| \geq \delta |B|$ and
  \begin{enumerate}
  \item the pair $A',B'$ is $R$-homogeneous,
  \item there are some $c_1 \in A^{|z_1|}$ and $c_2 \in B^{|z_2|}$ such that $A' = \psi_1(A,c_1)$
    and $B' = \psi_2(B,c_2)$.
  \end{enumerate}
\end{cor}
\begin{proof}
  Let $\psi_1(x,z_1), \psi_2(x,z_2), \varepsilon$ be as given by Corollary~\ref{cor: density extra
    parameters}. Then the existence of $A',B'$ follows by defining $\mu(X)$ ($\nu(X))$ to be the
  normalized number of points in $X\cap A$ (resp., $X \cap B$). Such Keisler measures are always
  generically stable by Fact~\ref{ExamplesOfSmoothMeasures}.

  For Part (2), by Remark~\ref{rem: UDTFS} we can find some formulas $\psi'_1(x,z'_1)$ and
  $\psi'_2(x,z'_2)$ such that for any finite sets $A,B$ and $c_1, c_2$ there are some
  $c'_1 \in A^{|z'_1|}, c'_2 \in B^{|z'_2|}$ such that $\psi_1(A,c_1) = \psi'_1(A,c'_1)$ and
  $\psi'_2(B,c_2) = \psi'_2(B,c'_2)$.
\end{proof}

We will show in Section~\ref{sec: equiv to distality} that the most basic version of Corollary
\ref{BasicRamseyForCountingMeasures} characterizes distality. This is not the case however if we do
not require definability of the homogeneous subsets.

\begin{rem}
  \begin{enumerate}
  \item If every definable relation in $\CM$ satisfies the strong Erd\H{o}s-Hajnal Property and
    $\mathcal{N}$ is interpretable in $\CM$, then every definable relation in $\mathcal{N}$
    satisfies the strong Erd\H{o}s-Hajnal Property.
  \item Let $\CM$ and $\mathcal{N}$ be two structures in the same language and assume that
    $\mathcal{N}$ embeds into $\CM$. If all quantifier-free definable relations in $\CM$ satisfy the
    strong Erd\H{o}s-Hajnal Property, then all quantifier-free definable relations in $\mathcal{N}$
    satisfy the strong Erd\H{o}s-Hajnal Property as well.
  \end{enumerate}
\end{rem}

By the remark and Corollary~\ref{BasicRamseyForCountingMeasures} we have the following.
\begin{cor}
  If $\CM$ is distal and $\mathcal{N}$ is interpretable in $\CM$ (embeds into $\CM$), then all
  definable (resp., quantifier-free definable) relations in $\mathcal{N}$ satisfy the strong
  Erd\H{o}s-Hajnal Property.
\end{cor}

\begin{sample}\label{ex: non-distal structures satisfying *} The following relations satisfy the
  strong Erd\H{o}s-Hajnal Property.
  \begin{enumerate}

  \item Definable relations in an arbitrary algebraically closed field of characteristic $0$ (since
    $\mathrm{ACF}_0$ is interpretable in the distal theory of real closed fields $\mathrm{RCF}$).
  \item Definable (in the language $L_{\mathrm{div}}$) relations in an arbitrary non-trivially
    valued algebraically closed field of residue characteristic $0$ (since its theory
    $\mathrm{ACVF}_{0,0}$ is interpretable in the theory of real closed valued fields
    $\mathrm{RCVF}$ (see e.g. \cite[Corollary 6.3]{VCD1}), which is distal in view of Example~\ref{ex: RCVF}).
  \item Quantifier-free definable relations in an arbitrary field of characteristic $0$ (as it can
    be embedded into some model of $\mathrm{ACF}_0$).
  \item Quantifier-free definable (in $L_{\mathrm{div}}$) relations in an arbitrary valued field of
    equicharacteristic $0$ (as it can always be embedded into a model of $\mathrm{ACVF}_{0,0}$).
  \end{enumerate}
\end{sample}

Thus (see Remark~\ref{rem: strong EH implies EH}) we obtain many new families of graphs satisfying
the Erd\H{o}s-Hajnal conjecture.

\begin{rem}
  \begin{enumerate}
  \item Every relation satisfying the strong Erd\H{o}s-Hajnal property is NIP.
  \item If all definable relations on $\CM$ satisfy the Erd\H{o}s-Hajnal property then $\CM$ is NIP.
  \end{enumerate}
\end{rem}
\begin{proof}
  (1) If the relation $R(x,y)$ is not NIP, then for any finite bi-partite graph $G$ there are some
  $A \subseteq M^{|x|}, B \subseteq M^{|y|}$ such that $G$ is isomorphic to $(A,B,R \cap (A \times B))$. By the
  optimality of the bound $O(\log n)$ on the size of homogeneous subsets in arbitrary bi-partite
  graphs it follows that $R$ does not have the strong Erd\H{o}s-Hajnal property.

  (2) If the relation $R(x,y)$ is not NIP, let $R' \subseteq M^{|x|+|y|} \times M^{|x|+|y|}$ be
  defined by $R'(ab,cd) \iff R(a,d) \lor R(c,b)$. This is a symmetric relation such that for any
  finite graph $G$ there is some set $A \subseteq M^{|x|+|y|}$ such that $G$ is isomorphic to
  $(R',A)$ (see e.g. \cite[Lemma 2.2]{laskowski2003karp}). Again optimality of the logarithmic bound
  for arbitrary graphs implies that $R'$ does not have the Erd\H{o}s-Hajnal property.
\end{proof}

\section{Regularity lemma for distal hypergraphs}\label{sec: distal regularity}
\subsection{Regularity lemmas for restricted families of graphs}
Szemer\'edi's regularity lemma is a fundamental result in graph combinatorics with many versions and
applications in extremal combinatorics, number theory and computer science (see
\cite{komlos1996szemeredi} for a survey). In it's simplest form for bi-partite graphs, it can be
presented as following.

\begin{fact}If $\varepsilon>0$, then there exists $K=K\left(\varepsilon\right)$ such that: for any
  finite bi-partite graph $R\subseteq A\times B$, there exist partitions
  $A=A_{1}\cup\ldots\cup A_{k_1}$ and $B=B_{1}\cup\ldots\cup B_{k_2}$ into non-empty sets, and a set
  $\Sigma\subseteq\left\{ 1,\ldots,k_1\right\} \times\left\{ 1,\ldots,k_2\right\}$ with the
  following properties.

  \begin{enumerate}

  \item Bounded size of the partition: $k_1, k_2 \leq K$.

  \item Few exceptions:
    $\left|\bigcup_{\left(i,j\right)\in\Sigma}A_{i}\times
      B_{j}\right|\geq(1-\varepsilon)\left|A\times B\right|$

  \item $\varepsilon$-regularity: for all $\left(i,j\right)\in\Sigma$, and all
    $A'\subseteq A_{i},B'\subseteq B_{j}$, one has  
          $$ \left| \left| R \cap \left( A' \times B' \right) \right| - d_{ij} \left|A'\right| \left| B' \right| \right| \leq \varepsilon \left|A \right| \left|B \right|,$$
          where $d_{ij} = \frac{\left| R \cap \left( A_i \times B_j \right) \right|}{\left|A_i \right| \left| B_j \right|}$.
  \end{enumerate}

\end{fact}

In general the bound on the size of the partition $K$ is known to grow as an exponential tower of
height $\frac{1}{\varepsilon}$, and the result is less informative in the case of sparse graphs.
Recently several improved regularity lemmas were obtained in the context of definable sets in
certain structures or in restricted families of structures.
\begin{enumerate}
\item \cite{tao2012expanding} Algebraic graphs of bounded complexity in large finite fields
  (equivalently, definable graphs in pseudofinite fields): pieces of the partition are algebraic of
  bounded complexity, no exceptional pairs, stronger regularity. Some generalizations and
  simplifications were obtained in \cite{pillay2013remarks, garcia2014pseudofinite} and by
  Hrushovski (unpublished).

\item \cite{lovasz2010regularity} Graphs of bounded VC-dimension: density arbitrarily close to $0$
  or $1$, the size of the partition is bounded by a polynomial in
  $\left(\frac{1}{\varepsilon}\right)$.

  \begin{enumerate}
  \item \cite{malliaris2014regularity} Graphs without arbitrarily large half-graphs, corresponding to
    the case of stable graphs (no exceptional pairs).

  \item \cite{GromovEtAl, fox2015polynomial} Semialgebraic graphs of bounded complexity.
  \end{enumerate}
\end{enumerate}

We remark that the classes of structures in (1), 2(a) and 2(b) are orthogonal to each other. In the
next section we give a generalization of the case 2(b) to graphs definable in arbitrary distal
structures.  As remarked before, the stable and the distal cases present two extremal cases of
general NIP structures.

\subsection{Distal regularity lemma}

We work in a model $\CM$ of a distal theory $T$. We have sorts $S_1,\dots, S_k$ (i.e.\  definable
subsets of some powers of $\CM$) and a definable relation $R \subseteq S_1\times \dotsb \times S_k$.

\begin{ntn}
  \begin{enumerate}[(a)]
  \item Let $\vec S=S_1\times\dotsb\times S_k$.
  \item We call a subset $X\subseteq \vec S$ \emph{a rectangular subset} if it is of the form
    $X=X_1\times \dotsb X_k$.
  \item For $A\subseteq M$ and a finite set of formulas
    $\vec \Delta=\{ \Delta_i(x_i, y_i), i=1,\dotsc,k\}$, a rectangular subset
    $X=X_1\times \dotsb X_k$ is called $\vec \Delta$-definable over $A$ if each $X_i$ is a finite
    Boolean combination of sets from $\{ \Delta_i(x_i,a) \colon a\in A \}$.  (In fact we will need
    only conjunctions of $\Delta_i$ and their negations, i.e.\  partial $\vec \Delta$-types.)
  \item Given Keisler measures $\mu_i$ on each sort $S_i$, for a rectangular definable
    $X=X_1\times\dotsb\times X_k$ we set
    \[ \mu(X)=\mu_1(X_1)\cdot\mu_2(X_2)\cdot\dotsc\cdot\mu_k(X_k). \]
  \item By \emph{a rectangular definable partition of $\vec S$} we mean a finite partition $\CP$ of
    $\vec S$ consisting of rectangular definable sets.
  \item For rectangular definable partitions $\CP,\CP_1$ of $\vec S$ we write $\CP \sqsubset \CP_1$
    if $\CP$ refines $\CP_1$, namely for each $X\in \CP$ there is $Y\in \CP_1$ with $X\subseteq Y$.
  \item Given Keisler measures $\mu_i$ on each sort $S_i$, for a rectangular definable partition
    $\CP$ of $\vec S$, we define \emph{the defect} of $\CP$ to be
    \[ \dd(\CP) := \sum_{ \substack{X\in \CP \\ X \text{ is not $R$-homogeneous} }} \mu(X). \]
    Obviously, if $\CP_1 \sqsubset \CP$ then $\dd(\CP_1)\leq \dd(\CP)$.
  \end{enumerate}
\end{ntn}

\begin{prop}\label{prop: partitioned regularity} There is some constant 
  $c=c \left(R \right)$ such that: for any $\varepsilon > 0$ and any generically stable measures $\mu_i$ on
  $S_i$, for $i=1, \ldots, k$, there is a rectangular uniformly definable (in terms of $R$ and
  $\varepsilon$) partition $\CP$ of $\vec S$ with
  $|\CP| \leq \left( \frac{1}{\varepsilon} \right)^c $ and $\dd(\CP) \leq \varepsilon$.

\end{prop}
We give a proof of the proposition in several claims, essentially following the proof of
\cite[Theorem 1.3]{fox2015polynomial} but working with Keisler measures.

Using Proposition~\ref{prop: DensityRamseyHypergraphs} we know that the following holds.

\begin{claim}\label{thm:main}
  There is a constant $\delta=\delta(R)$, and formulas $\Delta_i(x_i, y_i), i=1,\dotsc,k$ such that
  for any generically stable measures $\mu_i$ on $S_i$ there are $a_i, i=1,\dotsc,k$, such that the
  sets $X_i\subseteq S_i$ defined by $\Delta_i(x_i,a_i)$ are $R$-homogeneous and
  $\mu_i(X_i)\geq \delta$.

\end{claim}

\begin{rem}
  Since we are going to keep track of the parameters used in $\Delta_i$ it is more convenient to
  assume that in the above claim $y_1=\dotsb=y_k=y$ and $a_1=\dotsb=a_k=a$. It can be always
  achieved by a concatenation of variables.
\end{rem}

We fix $\delta$ from the previous claim and let $\vec\Delta=\{ \Delta_i(x_i,y), i=\dotsc,k\}$, where
$\Delta_i$ are from the above claim.

\begin{claim}\label{claim:1}
  Let $X$ be a definable rectangular subset of $\vec S$ with $\mu(X) > 0$. Then there is some
  $a \in M^{|y|}$ and a rectangular set $Y$ which is $\vec \Delta$-definable over $\{ a\}$, such
  that $X\cap Y$ is $R$-homogeneous and $\mu(X\cap Y)\geq \delta^k\mu(X)$.
\end{claim}
\begin{proof}
  Apply Claim~\ref{thm:main} to measures $\mu_i$ relativized to the sets $X_i$.
\end{proof}

\begin{claim}\label{claim:2}
  Let $\CP$ be a rectangular partition of $\vec S$ which is $\vec \Delta$-definable over a finite set
  $A$. Then there is a rectangular partition $\CP_1$ which is $\vec \Delta$-definable over a finite set $A_1$
  with
  \begin{enumerate}
  \item $|\CP_1|\leq (k+1)|\CP|$,
  \item $|A_1| \leq |A|+|\CP|$,
  \item $\dd (\CP_1) \leq (1-\delta^k)\dd(\CP)$.
  \end{enumerate}
\end{claim}
\begin{proof} Let $X \in \mathcal{P}$ be non-homogeneous. We can partition it into  $(k+1)$ rectangular subsets, all of them $\vec \Delta$-definable using one extra parameter $a_X$, such that one of these subsets is $R$-homogeneous and is of measure at least $\delta^k \mu(X)$. Namely, if $X = \prod_{i=1}^{k} X_i$, by Claim \ref{claim:1} there are some $Y_i \subseteq X_i$ such that $Y= \prod_{i=1}^{k} Y_i$ is an $R$-homogeneous subset of $X$ with $\mu(Y) \geq \delta^k \mu(X)$,  and we take a partition of $X$ into $(k+1)$ pieces given by the sets $Y_1 \times \ldots Y_k$ and $X_1 \ldots \times X_{i-1} \times (X_i \setminus Y_i) \times Y_i \times \ldots \times Y_k$ for all $i=1, \ldots, k$.

Replacing each non-homogeneous $X \in \mathcal{P}$ with such a sub-partition we obtain $\CP_1$ satisfying the requirements.
\end{proof}

Thus, by induction on $n$ we can construct a rectangular partition $\CP_n$ of $\vec S$ which is 
$\vec \Delta$-definable over a finite set $A_n$ and such that
\begin{enumerate}
\item $\dd(\CP_n)\leq (1-\delta^k)^n$,
\item $|\CP_n|\leq (k+1)^n$,
\item $|A_n| \leq \sum_{j<n} (k+1)^j =\dfrac{ (k+1)^n -1}{(k+1)-1}\leq (k+1)^n$.
\end{enumerate}

In particular, given $\varepsilon >0$, using $(1)$ and $(2)$ above, after
$N=\dfrac{\log{\varepsilon}}{\log{(1-\delta^k)}}= \left( - \frac{1}{\log(1-\delta^k)} \right) \log{\frac{1}{\varepsilon}}$ steps we have
$\dd(\CP_N) \leq \varepsilon$ with $|\mathcal{P}_N|\leq (k+1)^N \leq 2^{kN} \leq (\frac{1}{\varepsilon})^{c}$, where
$c = - \frac{k}{\log(1 - \delta^k)}$ is a positive constants depending only on $R$.

This finishes the proof of Proposition~\ref{prop: partitioned regularity}.

\medskip

From the above $k$-partite version we obtain a regularity lemma for hypergraphs.
\begin{thm}[Distal regularity lemma]\label{thm: distal regularity lemma} Let $P \subseteq M^{d}$ be
  a definable set and $R(x_1, \ldots, x_k)$ with $|x_i| = d$ for all $1 \leq i \leq k$ be a
  definable relation. Then there is some constant $c = c(R)$ such that the following holds. 

  For any $\varepsilon > 0$ and for any generically stable measure $\mu$ on $P$, there is a partition
  $P = P_1 \cup \ldots \cup P_K$ with $K = O \left( (\frac{1}{\varepsilon})^c \right)$ such that
  $P_i$'s are uniformly definable (in terms of $R$ and $\varepsilon$) and
 $$\sum \mu(P_{i_1} ) \ldots \mu(P_{i_k}) \leq \varepsilon, $$
 where the sum is over all tuples $(i_1, \ldots, i_k)$ such that $(P_{i_1}, \ldots, P_{i_k})$ is not
 $R$-homogeneous.
\end{thm}
\begin{proof}
  Let $\CP_N, A_N, c, \vec \Delta$ be as given by the proof of Proposition~\ref{prop: partitioned regularity}
  for $S_i = P$ and $\mu_i = \mu$, for all $1 \leq i \leq k$.

  Using Fact~\ref{fac: PolyTypesNIP}, we obtain constants $c_2$ and $c_3$ depending only on $R$ and such that the number of
  $\vec \Delta$-types over any finite set $A$ is bounded by $c_2|A|^{c_3}$.  Finally, we partition
  $P$ into realizations of complete $\vec \Delta$-types over $A_N$, say $P = \bigcup_{i \leq K} P_i$.
  It follows that there will be at most $c_2 (\frac{1}{\varepsilon})^{c c_3}$ parts. It is easy to
  see that this partition of $P$ satisfies the homogeneity condition because the rectangular
  partition
  $$\CP := \left\{ P_{i_1} \times \cdots \times P_{i_k} : 1 \leq i_1, \ldots, i_k \leq K \right\}$$
  refines $\CP_N$ and $\dd(\CP_N) \leq \varepsilon$.
\end{proof}

\subsection{Finding definable equipartitions}
Normally in the conclusion of a regularity lemma one is able to choose parts of (approximately)
equal measure. We give a sufficient condition for this in the definable setting. To simplify some
expressions, given real numbers $r_1, r_2$ and $\varepsilon >0$, we write
$r_1 \approx^{\varepsilon} r_2$ to denote that $|r_1 - r_2| < \varepsilon$.

\begin{defn}\label{def: uniformly cutting finite sets}
  We say that a structure $\CM$ \emph{uniformly cuts finite sets} if for every formula $\phi(x,y)$
  and every $\varepsilon > 0$ there is a formula $\chi(x,z)$ such that for any sufficiently large
  finite set $A \subseteq M^{|x|}$, any $b \in M^{|y|}$ and any $0 \leq m \leq |\phi(A,b)|$ there is
  some $c \in M^{|z|}$ such that
  $\frac{|\phi(A,b) \cap \chi(A,c)|}{|\phi(A,b)|} \approx^{\varepsilon} \frac{m}{|\phi(A,b)|}$.
\end{defn}
Note that this is a property of $\Th(\CM)$.

\begin{sample}
  \begin{enumerate}
  \item Assume that there is a definable linear order $x<y$ on $\CM$. Then clearly $\CM$ uniformly
    cuts finite sets and $\chi$ in Definition~\ref{def: uniformly cutting finite sets} can be chosen
    independently of $\phi$ and $\varepsilon$ (using lexicographic ordering for subsets of $M^n$ for
    $n>1$).

  \item Let $\CM = (\mathbb{Q}_p,+,\times)$, then $\CM$ uniformly cuts finite sets. For subsets of
    $M$ this follows from the fact that every ball in $\mathbb{Q}_p$ is a disjoint union of exactly
    $p$ balls, using which an argument similar to the proof that for an atomless measure, every set
    of positive measure contains subsets of arbitrary smaller measure, can be carried out up to
    $\varepsilon$, in a number of steps bounded in terms of $\varepsilon$ (one can check using
    quantifier elimination in the $p$-adics that in this case $\chi$ cannot be chosen independently of
    $\varepsilon$). To extend this to subsets of $M^n$ for $n>1$, note that if $k$ is an infinite
    field and $A\subseteq k^n$ is a finite set then there is a uniformly definable linear map $f\colon k^n\to k$ that is
    one-to-one on $A$.

  \end{enumerate}
\end{sample}

\begin{prop}\label{prop: cutting gen stable measures} Let $\CM$ be a distal structure and assume
  that it uniformly cuts finite sets. Then for every formula $\phi(x,y)$ and $\delta > 0$ there is
  some $\chi(x,z)$ such that: for any generically stable measure $\mu$ on $\CM$ with
  $\mu(\{ c \}) = 0$ for any singleton $c \in M^{|x|}$, if $ 0\leq \alpha \leq \mu(\phi(x,a))$ then
  we can find some $b \in M^{|z|}$ with $\mu(\phi(x,a) \cap \chi(x,b)) \approx^{\delta} \alpha$.
\end{prop}

\begin{proof}  
  Fix $\varepsilon > 0$ arbitrary, and let $\chi(x,z)$ be an arbitrary formula. As $\CM$ is distal,
  it follows by Fact~\ref{fac: distal iff gen stable is smooth} that $\mu$ is smooth over $M$. Let
  $\theta_i^1(x), \theta_i^2(x), \psi_i(y,z)$, $i=1, \ldots, n$ list all of the formulas over $M$
  given by Fact~\ref{non-uniform epsilon-net for smooth} for each of $\phi(x,y)$ and
  $\phi(x,y) \land \chi(x,z)$, with respect to $\mu$ and $\varepsilon$.  Let $\mathcal{B}$ be the
  finite Boolean algebra of subsets of $M^{|x|}$ generated by
  $\Theta = \{ \theta^t_i(M) : i=1, \ldots, n, t=1,2\}$. Clearly the number of atoms in $\mathcal{B}$ is at most $4^{n}$. By
  assumption every definable set of positive $\mu$-measure is infinite. Then for all sufficiently
  large $m \in \mathbb{N}$ we can choose a set $C \subseteq M^{|x|}, |C|=m$ such that for every atom
  $A$ of $\mathcal{B}$, $| \frac{|C \cap A|}{m} - \mu(A) | < \frac{\varepsilon}{2 \cdot 4^n}$. It
  then follows that for every $i \in \{1, \ldots, n\}, t \in \{1,2 \}$ we have
  $ \frac{|\theta^t_i(C)|}{|C|} \approx^{\frac{\varepsilon}{2}} \mu(\theta^t_i(M))$. But by the
  choice of $\theta^0_i, \theta^1_i$ this implies that for any set $D$ from
  $\Delta(M) = \{ \phi(M,a) : a \in M \} \cup \{ \phi(M,a) \cap \chi(M,b) : a, b \in M \}$ we have
  $\frac{|D \cap C|}{|C|} \approx^{\varepsilon} \mu(D)$.

  Now let $0<\alpha < \beta := \mu(\phi(M,a))$ be given (if $\alpha \in \{ 0, \beta \}$ then there
  is nothing to do). Let $\varepsilon := \frac{\delta}{4}$, and let $\chi(x,z)$ be as given by
  Definition~\ref{def: uniformly cutting finite sets} for $\phi(x,y)$ and $\varepsilon$. Take $m$
  sufficiently large (to be specified later), then for $C$ with $|C| = m$ chosen as above with
  respect to $\varepsilon$ and $\chi$ we have in particular
  $\frac{|\phi(C,a)|}{|C|} \approx^{\varepsilon} \beta$. Let $l := |\phi(C,a)|$. We may assume that
  there is some $k \in \mathbb{N}, k \leq l$ such that
  $\frac{k}{l} \approx^{\varepsilon} \frac{\alpha}{ \beta}$ (as $\alpha > 0$, by choosing $m$
  sufficiently large we may assume that $l$ is arbitrarily large). Then, using that both $\beta, \frac{k}{l} \leq 1$ we obtain 
  $\alpha \approx^{\varepsilon} \beta \frac{k}{l} \approx^{\varepsilon} \frac{l}{m} \frac{k}{l} =
  \frac{k}{m}$, so $\alpha \approx^{2 \varepsilon} \frac{k}{m}$.
  
  By the choice of $\chi(x,z)$, there is some $b \in M^{|z|}$ such that
  $\frac{|\phi(C,a) \cap \chi(C,b)|}{l} \approx^{\varepsilon} \frac{k}{l}$, which implies
  $\frac{|\phi(C,a) \cap \chi(C,b)|}{m} \approx^{\varepsilon} \frac{k}{m}$, and so
  $\frac{|\phi(C,a) \cap \chi(C,b)|}{m} \approx^{3 \varepsilon} \alpha$. By the assumption on $C$
  this implies that $\mu(\phi(x,a)\land \chi(x,b)) \approx^{4 \varepsilon} \alpha$,
  i.e. $\mu(\phi(x,a)\land \chi(x,b)) \approx^{\delta} \alpha$.

\end{proof}

\begin{rem}\label{rem: definable measure def}
	Recall that a global measure $\mu$ is \emph{definable} over a small model $M$ if it is $\Aut(\UU/M)$-invariant and for every formula $\phi(x,y) \in L$ and every closed subset $X$ of $[0,1]$,  the set $\{ q \in S_{|y|}(M) : \mu(\phi(x,b)) \in X \textrm{ for any } b\in \UU^{|y|}, b\models q(y)\}$ is closed. It is \emph{finitely satisfiable} if for every $\phi(x,b) \in L(\UU)$ with $\mu(\phi(x,b)) > 0$ there is some $a \in M^{|x|}$ such that $\models \phi(a,b)$ holds. As mentioned before, in an NIP structure, a Keisler measure $\mu$ over $\mathbb{M}$ is generically stable if and only if it admits a global $M$-invariant extension which is both  definable over $M$ and finitely satisfiable in $M$ (see \cite[Theorem 3.2]{NIP3}).
\end{rem}

\begin{cor}
  Assume that $T$ uniformly cuts finite sets in such a way that $\chi$ in Definition~\ref{def:
    uniformly cutting finite sets} can be chosen independently of $\varepsilon$ (e.g. if $\mathcal{M}$ has a definable linear order). Then under the assumptions of 
  Proposition~\ref{prop: cutting gen stable measures} we can choose $b \in \UU$ such that
  $\mu_1(\phi(x,a)\cap \chi(x,b)) = \alpha$, where $\mu_1$ is the unique global Keisler measure
  extending $\mu$.
\end{cor}
\begin{proof}
  As $\mu_1$ is generically stable over $\CM$, it is in particular definable over $\CM$. That is,
  for every $\delta > 0$ the set
  $\{ b \in \UU : \alpha - \delta \leq \mu_1(\phi(x,a) \land \chi(x,b)) \leq  \alpha + \delta \}$ is
  type-definable over $\CM$ (and consistent). It then follows by compactness that we can find some
  $b^* \in \UU$ with $\mu_1(\phi(x,a) \land \psi(x,b)) = \alpha$.
\end{proof}

\begin{cor} Let $\CM$ be a distal structure and assume that it uniformly cuts finite sets. Then in
  Theorem~\ref{thm: distal regularity lemma} for any $\mu$ satisfying in addition $\mu(\{ c \}) = 0$
  for all $c \in M^d$ and any $\delta > 0$ we can find a partition $P_1, \ldots, P_K$ with
  $\mu(P_i) \approx^{\delta} \mu(P_j)$ for all $1 \leq i,j \leq K$ (and the parts $P_i$ are
  uniformly definable in terms of $R, \varepsilon, \delta, \chi$).
\end{cor}

\begin{proof} 
  We are following the standard repartition argument (see e.g.\  \cite[Proof of Theorem
  1.3]{fox2015polynomial}).
    
  Let $(P,R)$ be a $k$-uniform hypergraph, and let $P = P_1 \cup \ldots \cup P_K$ be a partition of
  its vertices given by Theorem~\ref{thm: distal regularity lemma} for $\frac{\varepsilon}{2}$, with
  $K \leq c_1 \left( \frac{2}{\varepsilon} \right)^{c_2}$. Fix $\delta > 0$ and $\mu$ satisfying the
  assumptions, and we will find a new partition $P = Q_1 \cup \ldots \cup Q_{K'}$ satisfying the
  conclusion of the corollary for $\varepsilon$ and $\delta$.

  Let $K' = \lceil 4 \frac{2^k}{\varepsilon}K \rceil $, without loss of generality
  $0 < \delta < \frac{1}{K'}$, and fix an arbitrary $0 < \delta' < \frac{\delta}{K'}$.  Using
  Proposition~\ref{prop: cutting gen stable measures} we can partition each $P_i$ into
  $P_i = S_i \cup \bigcup Q_{i,j}$ with $\mu(Q_{i,j}) \approx^{\delta'} \frac{1}{K'}$ for all $j$
  and the remainder $\mu(S_i) < \frac{1}{K'}$. Let now $S = \bigcup_{i} S_i$, and again using
  Proposition~\ref{prop: cutting gen stable measures} we can partition $S$ into
  $S = T \cup \bigcup U_j$ with $\mu(U_j) \approx^{\delta'} \frac{1}{K'}$ and the remainder
  $\mu(T) < \frac{1}{K'}$. As $\delta'$ was sufficiently small compared to $\delta$ and
  $\frac{1}{K'}$, calculating the error we get $\mu(T) \approx^{\delta} \frac{1}{K'}$. We claim that
  $P = \bigcup Q_{i,j} \cup \bigcup U_j \cup T$ is the required partition, re-enumerate it as
  $P = Q_1 \cup \ldots \cup Q_{K'}$. We still have that $K'$ is a polynomial in
  $\frac{1}{\varepsilon}$. Note that $\mu(S) < \frac{K}{K'}$, so there are at most $K$ parts of the
  new partition contained in $S$. Hence the sum $\sum \mu(Q_{i_1}) \ldots \mu(Q_{i_k})$ over all
  tuples $(i_1, \ldots, i_k)$ for which not all of $P_{i_1}, \ldots, P_{i_k}$ are subsets of parts
  of the original partition is at most
  $K {\left( K' \right)}^{k-1} \left( \frac{1}{K'} + \delta \right)^{k} \leq K \left( K' \right)^{k-1}
  \frac{2^k}{\left(K' \right)^k} = \frac{2^k K}{K'} = \frac{\varepsilon}{4}$.
  Together with the assumption on the original partition it then follows that
  $\sum \mu(Q_{i_1} ) \ldots \mu(Q_{i_k})$, where the sum is over all tuples $(i_1, \ldots, i_k)$
  such that $(Q_{i_1}, \ldots, Q_{i_k})$ is not $R$-homogeneous, is bounded by
  $\frac{\varepsilon}{2} + \frac{\varepsilon}{4} < \varepsilon$.

  It follows from the construction that the new partition is uniformly definable in terms of the old
  one and $\delta$, and thus uniformly definable in terms of $R, \varepsilon$ and $\delta$.
\end{proof}

\begin{rem}Answering a question from an earlier version of our article, Pierre Simon had recently demonstrated that all distal structures uniformly cut finite sets.
\end{rem}

\section{Equivalence to distality}\label{sec: equiv to distality}

\subsection{Strong Erd\H{o}s-Hajnal fails in $ACF_p$}\label{sec: Failure in ACFp}

Fields of positive characteristic give a standard example of the failure of the strong
Szemer\'edi-Trotter bound on the number of incidences between points and lines. In a personal
communication Terrence Tao had suggested that it may also be used for the failure of the strong
Erd\H{o}s-Hajnal property in this setting, which turned out to be the case indeed.

Let $\mathbb{F}$ be a field. For a set of points $P \subseteq \mathbb{F}^2$ and a set of lines $L$ in
$\mathbb{F}^2$ we denote by $I(P,L) \subseteq P \times L$ the incidence relation, i.e.
$I(P,L) = \{ (p,l) \in P \times L : p \in l \}$.  As remarked in Example~\ref{ex: non-distal
  structures satisfying *}, every field $\mathbb{F}$ of characteristic $0$ satisfies the strong
Erd\H{o}s-Hajnal property with respect to quantifier-free formulas. In particular we have:

\begin{prop}\label{prop: PointsLines} Let $\mathbb{F}$ be a field of characteristic $0$. Then there
  is a constant $\delta > 0$ such that for any finite (sufficiently large) set of points
  $P \subseteq \mathbb{F}^2$ and any finite (sufficiently large) set of lines $L$ in $\mathbb{F}^2$
  there are some $P_0 \subseteq P$ and $L_0 \subseteq L$ with
  $|P_0| \geq \delta |P|, |L_0| \geq \delta |L|$ and $I(P_0, L_0) = \emptyset$.
\end{prop}

We show that the assumption of characteristic $0$ cannot be removed.

\begin{prop} We fix a prime $p$, and let $\mathbb{F} = \mathbb{F}_p^{alg}$. Then the conclusion of
  Proposition~\ref{prop: PointsLines} fails in $\mathbb{F}$.
\end{prop}
\begin{proof}
  Assume towards a contradiction that $\mathbb{F}$ satisfies Proposition~\ref{prop: PointsLines}.

  Since every finite field of characteristic $p$ can be embedded into $\mathbb{F}$, we obtain that
  the following would be true:

  Let $\mathbb{F}_q$ be a finite field of characteristic $p$, of size $q$. Let $P$ be the set of
  \emph{all} points in $\mathbb{F}_q$, and let $L$ be the set of \emph{all} lines in
  $\mathbb{F}_q^2$ of the form $y = ax + b$. Then there are $P_0 \subseteq P, L_0 \subseteq L$ with
  $|P_0| \geq \delta |P|, |L_0| \geq \delta |L|$ and such that $I(P_0,L_0) = \emptyset$.

  We show that this is impossible. We have $|P| = q^2, |L| = q^2$. Since $\mathbb{F}_q$ has size
  $q$, every line contains exactly $q$ points, therefore $|I(P,L)| = |L|q = q^3$. Notice also that
  every point belongs to exactly $q$ lines in $L$.

  We fix $k$ large enough so that $\frac{1}{p^k} < \delta$, and let $\delta_0 =
  \frac{1}{p^k}$. Since $q = p^n$ for some $n$, $\delta_0 q$ is an integer for every $q \geq p^k$.

  Hence we can choose $P_0 \subseteq P$ with $|P_0| = \delta_0 |P|$ and $L_0 \subseteq L$ with
  $|L_0| = \delta_0 |L|$ such that $I(P_0, L_0) = \emptyset$. Let
  $P_1 = P \setminus P_0, L_1 = L \setminus L_0$. We have $|P_1| = (1 - \delta_0) q^2$ and
  $|L_1| = (1 - \delta_0) q^2$.

  Consider $I(P_0, L)$. Since every point belongs to exactly $q$ lines in $L$ we have
  $|I(P_0, L)| = |P_0| q = \delta_0 q^3$. Since $I(P_0, L_0) = \emptyset$, we have
  $I(P_0,L) = I(P_0, L_1)$, so $|I(P_0,L_1)| = \delta_0 q^3$.

  On the other hand,  from the Cauchy-Schwartz inequality (see e.g. \cite[Page 1]{Rudnev}) it follows
  that $|I(P_0, L_1)| \leq \sqrt{|L_1|} \sqrt{|I(P_0, L_1)| + |P_0|^2}$.

  Thus we have
$$\delta_0 q^3 \leq \sqrt{(1 - \delta_0)q^2} \sqrt{\delta_0 q^3 + \delta_0^2 q^4} = \sqrt{(1-\delta_0)\delta_0 q^5 + (1-\delta_0) \delta_0^2 q^6} \textrm{.}$$
 
Since $(1 - \delta_0) < 1$, the above inequality fails for large enough $q$ --- a contradiction.

\end{proof}

\begin{cor}
  Let $K$ be an infinite field definable in a distal structure $\CM$. Then $char(K) = 0$.
\end{cor}
\begin{proof}
  By \cite[Corollary 4.5]{kaplan2011artin} every infinite NIP field of characteristic $p>0$ contains
  $\mathbb{F}_p^{alg}$. But then $\CM$ cannot satisfy the strong Erd\H{o}s-Hajnal property by the
  proposition above, contradicting distality.
\end{proof}

In particular the theory $ACF_p$ admits no distal expansion. No examples of NIP theories with this
property were known until now.

\subsection{Equivalence to distality}\label{sec: equivalence to distality} In this section we
assume some familiarity with NIP theories (see e.g. \cite{SimBook}) and recall some facts about
distal theories. We fix a theory $T$ and a big sufficiently saturated model $\UU$ of $T$. Recall
that a sequence $(a_i : i \in I)$ of elements of $M^n$ indexed by a linear order $I$ is
\emph{indiscernible} over a set of parameters $A \subseteq M$ if for any $i_1 < \ldots < i_k$ and
$j_1 < \ldots < j_k$ from $I$ we have
$\tp(a_{i_1}, \ldots, a_{i_k}/A) = \tp(a_{j_1}, \ldots, a_{j_k}/A)$. Given a linear order $I$, by a
\emph{Dedekind cut} in $I$ we mean a cut $I = I_1 + I_2$ (i.e.,
$I = I_1 \cup I_2, I_1 \cap I_2 = \emptyset$ and $a < b$ for all $a \in I_1, b \in I_2$) such that
$I_1$ has no maximal element and $I_2$ has no minimal element. We denote by $I^*$ the reverse of the
order on $I$.

\begin{fact}\label{fac: distality indiscernibles}\cite{ExtDefII} Let $T$ be a complete NIP
  theory. Then the following are equivalent:
  \begin{enumerate}
  \item $T$ is distal (in the sense of Fact~\ref{StrongHonestDef}).
  \item Every indiscernible sequence $I \subseteq M^d$ in any model of $M$ of $T$ is distal. That
    is, for any two distinct Dedekind cuts of $I$, if some two elements fill them separately, then
    they also fill them simultaneously: if $I = I_1 + I_2 + I_3$ and we have some $a$ and $b$ from
    $M^d$ such that both $I_1 + a + I_2 + I_3$ and $I_1 + I_2 + b + I_3$ are indiscernible, then
    $I_1 + a + I_2 + b + I_3$ is indiscernible.
  \end{enumerate}
\end{fact}

Given an indiscernible sequence $I = (a_i : i \in [0,1])$, one defines the \emph{average measure}
$\mu$ of $I$ as the global Keisler measure given by
$\mu(\phi(x)) = \lambda_1(\{ i \in [0,1] : a_i \models \phi(x) \})$ for all definable sets
$\phi(x)$, where $\lambda_1$ is the Lebesgue measure on $[0,1]$. It follows from NIP that this
Keisler measure is well-defined, i.e.\  that the corresponding set of indices is measurable for every
$\phi(x)$ with parameters from $\UU$. We say that $\mathfrak{c} = (I_1, I_2, t)$ is a
\emph{polarized} cut of $I$ if $I = I_1 + I_2$ is a cut of $I$ and $t \in \{ 1, 2 \}$ specifies
whether it is approached from the left or from the right. It follows from NIP that for a polarized
Dedekind cut $\mathfrak{c} = (I_1, I_2, t)$ and a set of parameters $A \subseteq \UU$ we have a
complete \emph{limit type} of $\mathfrak{c}$ over $A$ denoted by $\lim(\mathfrak{c}/A)$ and defined
by $\phi(x) \in \lim(\mathfrak{c}/A) \iff$ the set $\{ i \in I_t : \UU \models \phi(a_i) \}$ is
unbounded from above in case $t = 1$, or from below in case $t=2$.

\begin{fact}\label{AverageMeasures}
  Let $T$ be NIP, let $I$ be an indiscernible sequence and let $\mu$ be the average measure of $I$.
  \begin{enumerate}
  \item The measure $\mu$ is generically stable (\cite[Proposition 3.7]{NIP3}).
  \item The support of $\mu$ is exactly the set of limit types of cuts of $I$. That is, if for some
    formula $\phi(x)$ we have $\mu(\phi(x)) > 0$ then $\phi(x) \in \lim(\mathfrak{c}/\UU)$ for some
    polarized Dedekind cut $\mathfrak{c}$ of $I$ \cite[Lemma 2.20]{Distal}.
  \item $I$ is distal if and only if $\mu$ is smooth \cite[Proposition 2.21]{Distal}.
  \end{enumerate}
\end{fact}
Recall that a sequence $(a_i : i \in I)$ is \emph{totally indiscernible} if for any
$i_1 \neq \ldots \neq i_k$ and $j_1 \neq \ldots \neq j_k$ from $I$ we have
$\tp(a_{i_1}, \ldots, a_{i_k}/A) = \tp(a_{j_1}, \ldots, a_{j_k}/A)$
\begin{fact}\label{fac: dp-min distal totally indisc} \cite{Distal} Let $T$ be dp-minimal. Then it
  is distal if and only if no infinite indiscernible sequence is totally indiscernible.
\end{fact}
\begin{rem}\label{rem: dp-min dist exp} It follows that if $T'$ is a dp-minimal expansion of a
  distal dp-minimal theory $T$, then $T'$ is distal as well.

  Indeed, If $T'$ expands $T$ and $(a_i : i \in I)$ is an infinite $L'$-indiscernible sequence, then
  it is in particular $L$-indiscernible, so not totally-$L$-indiscernible by distality of $T$, so of
  course not totally-$L'$-indiscernible.
\end{rem}

\begin{fact}[Strong base change, Lemma 2.8 in \cite{Distal}]\label{StrongBaseChange} Let $T$ be
  NIP. Let $I$ be an indiscernible sequence and $A \supseteq I$ a set of parameters. Let
  $(\mathfrak{c}_i : i < \alpha)$ be a sequence of pairwise-distinct polarized Dedekind cuts in
  $I$. For each $i$, let $d_i$ fill the cut $\mathfrak{c}_i$ (i.e., if
  $\mathfrak{c}_i = (I_1,I_2,t) $ then $I_1 + d_i + I_2$ is indiscernible). Then there exist
  $(d'_i : i < \alpha)$ in $\mathbb{U}$ such that:
  \begin{enumerate}
  \item $\tp((d'_i)_{i<\alpha} / I) = \tp((d_i)_{i<\alpha} / I)$,
  \item for each $i<\alpha$, $\tp(d'_i / A) = \lim (\mathfrak{c}_i / A)$.
  \end{enumerate}
\end{fact}

Finally, we will use the following finitary version of a characteristic property of NIP theories.

\begin{fact}\label{fac: NIP finite alternation} Let $\phi(x,y)$ be an NIP formula. Then there are
  some $k,N \in \mathbb{N}$ such that for any indiscernible sequence $I = (a_i : i < n)$ from
  $M^{|x|}$ with $n \geq N$ and any $b \in M^{|y|}$, the set $\phi(I,b)$ is a disjoint union of at
  most $k$ intervals.
\end{fact}
\begin{proof}
  Follows from the usual characterization of NIP via bounded alternation on indiscernible sequences
  (see e.g. \cite[Proposition 4]{AdlerNIP}) plus compactness.
\end{proof}

\begin{thm}\label{thm: characterizing distality} Let $T$ be an NIP theory. The following are
  equivalent:
  \begin{enumerate}
  \item $T$ is distal.
  \item For any definable relation $R(x,y)$ and any global generically stable measures
    $\mu_1, \mu_2$ there are some definable $X \subseteq \UU^{|x|}, Y \subseteq \UU^{|y|}$ which are
    $R$-homogeneous and satisfy $\mu_1(X) > 0, \mu_2(Y) > 0$.
  \item For any definable relation $R(x,y)$ there is some $\delta>0$ and some formulas
    $\psi_1(x,z_1), \psi_2(x,z_2)$ such that for all finite
    $A \subseteq \UU^{|x|}, B \subseteq \UU^{|y|}$ there are some $c_i \in \UU^{|z_i|}, i=1,2$ such
    that $|\psi_1(A,c_1)| \geq \delta |A|, |\psi_2(B,c_2)| \geq \delta |B|$ and the pair of sets
    $\psi_1(A,c_1), \psi_2(B,c_2)$ is $R$-homogeneous.
  \end{enumerate}
\end{thm}

\begin{proof}

  (1) implies (2) and (3) follow from Corollaries~\ref{DensityForGraphs} and
  \ref{BasicRamseyForCountingMeasures}.\\

  (2) implies (1). Assume that $I=(a_i)_{i \in \mathcal{I}}$ is a non-distal indiscernible sequence,
  with $\mathcal{I} = [0,1]$. This means that $I$ can be written as $I = I_1 + I_2 + I_3$ (where
  $I_j = (a_i : i \in \mathcal{I}_j)$ and
  $\mathcal{I}_1, \mathcal{I}_2^*, \mathcal{I}_2, \mathcal{I}_3^*$ are without last elements) in
  such a way that there are some $c,d \in \UU$ such that $I_1+c+I_2+I_3$ and $I_1+I_2+d+I_3$ are
  indiscernible, but $I_1+c+I_2+d+I_3$ is not.

  Then there is a formula $\phi(I'_{1}, x, I'_2, y, I'_3)$ with some finite $I'_j \subset I_j$, say
  $I'_j = (a_i : i \in \mathcal{I}'_j), \mathcal{I}'_j \subset \mathcal{I}_j$ for
  $j\in \{ 1, 2, 3 \}$, such that $\phi(I'_1, a_j, I'_2, a_k, I'_3)$ holds for any
  $\mathcal{I}'_1 < j < \mathcal{I}'_2 < k < \mathcal{I}'_3$, but
  $\models \neg \phi(I'_1, c, I'_2, d, I'_3)$. Let $[j_1, j_2]$ be some interval of $\mathcal{I}$
  between $\mathcal{I}'_1$ and $\mathcal{I}'_2$, and $[k_1, k_2]$ some interval between
  $\mathcal{I}'_2$ and $\mathcal{I}'_3$. Let
  $J = (a_i : i \in [j_1, j_2]), K=(a_i : i \in [k_1,k_2])$. Let $\mu$ be the average measure of
  $J$, and $\nu$ the average measure of $K$ (we may assume that both sequences are indexed by
  $[0,1]$ by taking an order preserving bijection). Then both $\mu$ and $\nu$ are generically stable
  by Fact~\ref{AverageMeasures}(1).

  Now assume that $X = \xi(\UU)$ and $Y = \chi(\UU)$ are definable subsets of $\UU^{|x|}$ with
  $\mu( X ) >0$ and $\nu( Y ) > 0$, where $\xi, \chi$ are formulas with parameters in some small
  model $M \supseteq I$. By Fact~\ref{AverageMeasures}(2) it follows that there is some polarized
  Dedekind cut $\mathfrak{c}$ of $J$ such that $\xi(x) \in \lim_J (\mathfrak{c}/M)$, and some
  polarized Dedekind cut $\mathfrak{d}$ of $K$ such that $\chi(x) \in \lim_K (\mathfrak{d}/M)$.

  It follows by compactness, indiscernibility of $I$ and taking an automorphism of $\UU$ that there
  is some $c'$ filling $\mathfrak{c}$ and $d'$ filling $\mathfrak{d}$ (separately, as cuts in $I$)
  such that $\neg \phi(I'_1, c', I'_2, d', I'_3)$ holds.  By Fact~\ref{StrongBaseChange} we can find
  some $c'', d''$ such that still $\neg \phi(I'_1, c'', I'_2, d'', I'_3)$ holds, but moreover
  $c'' \models \lim(\mathfrak{c}/M), d'' \models \lim(\mathfrak{d}/M)$. In particular,
  $\models \xi(c'') \land \chi(d'')$. On the other hand, by the choice of $\phi$ and the definition
  of $\mu, \nu$ there are some $j<k$ in $I$ such that
  $\models \xi(a_j) \land \chi(a_k) \land \phi(I'_1, a_j, I'_2, a_k, I'_3)$. This shows that the
  relation $R(x,y) = \phi(I'_1, x, I'_2, y, I'_3)$ is not homogeneous on $X \times Y$. As $X,Y$ were
  arbitrary definable sets of positive measure, we conclude.\\

  (3) implies (1).  Assume that $T$ is not distal, and we will show that (3) cannot hold. Working in
  $\UU$ we have some $I_i = (a^i_j : j \in \mathbb{Q})$ for $i\in \{ 1,2,3 \}$ and $a,b$ such that
  $I = I_1 + I_2 + I_3$, $I_1+a+I_2 + I_3$ and $I_1 + I_2 + b + I_3$ are indiscernible, but
  $I_1 + a +I_2 + b + I_3$ is not. This implies in particular that there is a formula $\phi \in L$
  such that $\models \neg \phi ( J'_1, a, J'_2, b, J'_3)$ for some finite $J'_i \subset I_i$ with
  $J'_i = (a_j : j \in \mathcal{J}_j)$, but $\models \phi (J'_1, a', J'_2, b', J'_3)$ for any
  $a', b' \in I$ such that $\mathcal{J}'_1 < a' < \mathcal{J}'_2 < b' < \mathcal{J}'_3$.

  Let now $R(x,y; c) := \phi(J'_1,x,J'_2,y,J'_3)$ with $c := J'_1 J'_2 J'_3$.  Assume that there are
  $\psi_i(x,y), i \in \{ 1, 2 \}$ and $\delta > 0$ as required by (3) for $R$. As $T$ is NIP, it
  follows by Fact~\ref{fac: NIP finite alternation} that there are some $k,N \in \omega$ such that
  for any indiscernible sequence $K= (a_j : j<n)$ with $n \geq N$ and any
  $d_i \in \UU, i \in \{ 1,2 \}$, the set $\psi_i(K,d_i)$ is a disjoint union of at most $k$
  intervals. Without loss of generality it then follows from (3) that there is some $k' \in \omega$
  such that for any finite indiscernible sequences $A=(a_j : j < n)$ and $B= (b_j : j < n)$ with
  $n \geq N$ we can find \emph{intervals} $A_0 \subseteq A, |A_0| \geq \frac{|A|}{k'}$ and
  $B_0 \subseteq B, |B_0| \geq \frac{|B|}{k'}$ such that $(A_0,B_0)$ is $R(x,y;c)$-homogeneous. We
  are going to show that this property fails.

  Re-enumerating the sequence we may assume that $I_1 = I_{1,0}+I_{1,1} + \ldots$ and
  $I_3 = \ldots + I_{3,1} + I_{3,0}$, with each of $I_{i,j}$ indexed by $\mathbb{Q}$, and that
  $J_1' \subset I_{1,0}, J'_3 \subset I_{3,0}$. Let $I'_1 := I_1 \setminus I_{1,0}$,
  $I'_3 = I_3 \setminus I_{3,0}$.

  By indiscernibility of $I$, automorphism and compactness for any Dedekind cuts $\frak{c}$ of
  $I'_1$ and $\frak{c}'$ of $I'_3$ we can find some $a'$ and $b'$ which fill those cuts (separately,
  viewed as cuts in $I$) and such that $\models \neg \phi ( J'_1, a', J'_2, b', J'_3)$ holds.

  For each $i \in \omega$, let $(\frak{c}_{i,j} : j \in \omega)$ be an infinite increasing sequence
  of cuts of $I_{1,i}$, and let $(\frak{c}'_{i,j} : j \in \omega)$ be a decreasing sequence of cuts
  of $I_{3,i}$. By the previous remark, let $a_{i,j}$ and $b_{i,j}$ be such that $a_{i,j}$ fills the
  cut $\frak{c}_{i,j}$, $b_{i,j}$ fills the cut $\frak{c}'_{i,j}$ and
  $\models \neg \phi ( J'_1, a_{i,j}, J'_2, b_{i,j}, J'_3)$ holds.

  Next using Fact~\ref{StrongBaseChange} and induction we can choose $a'_{i,j}, b'_{i,j}$ such that:
  \begin{itemize}
  \item $\tp(a'_{i,j} b'_{i,j}/ I) = \tp(a_{i,j} b_{i,j} / I)$,

  \item $\tp(a'_{i,j}/I A_{i,j}) = \lim(\frak{c}_{i,j}/ I A_{i,j})$, where
    $A_{i,j} = \{ a'_{i,j'} : j' < j \} \cup \{a'_{i',j'}: i'<i, j' \in \omega \}$,

  \item $\tp(b'_{i,j}/I B_{i,j}) = \lim(\frak{c}'_{i,j}/ I B_{i,j})$, where
    $B_{i,j} = \{ b'_{i,j'} : j' < j \} \cup \{b'_{i',j'}: i'<i, j' \in \omega \}$.

  \end{itemize}

  From this we have:
  \begin{enumerate}
  \item[(a)] For any $i,j \in \omega$ we have that
    $\models \neg \phi ( J'_1, a'_{i,j}, J'_2, b'_{i,j}, J'_3)$ holds.
  \item[(b)] The sequence $I'_1$ with all the $\{ a'_{i,j} : i,j \in \omega \}$ added in the
    corresponding cuts is an indiscernible sequence,
  \item[(c)] The sequence $I'_3$ with all the $\{ b'_{i,j} : i,j \in \omega \}$ added in the
    corresponding cuts is an indiscernible sequence,
  \item[(d)] $\models \phi (J_1, a', J_2, b', J_3)$ holds for any $a' \in I'_1, b' \in I'_3$.
  \end{enumerate}

  Here (a) follows from the first bullet and the choice of $a_{i,j}, b_{i,j}$; using the second
  bullet above it is easy to show that (b) holds, and that the sequence has the same EM-type as $I$
  (similarly for (c)); (d) was already observed above.

  In view of (a)--(d) above, for any $m \in \omega$ we can choose indiscernible sequences
  $A=(a_j : j < 2 k' m)$ and $B=(b_j : j < 2 k' m)$ such that for any $l_1, l_2 < 2k'$ we have
  $\models \neg R(a_{l_1m + l_2}, b_{l_2m+l_1};c)$ and $\models R(a_{l_1m+j_1}, b_{l_2m + j_2};c)$
  for any $j_1, j_2 \in (2k',m)$. It then follows that for all sufficiently large $m$, for any
  choice of an interval $A_0 \subseteq A$ with $|A_0| \geq \frac{|A|}{k'} \geq 2m$ and
  $B_0 \subseteq B$ with $|B_0| \geq \frac{|B|}{k'} \geq 2m$, the sets $(A_0, B_0)$ cannot be
  $R(x,y;c)$-homogeneous --- a contradiction to the choice of $k'$.
\end{proof}

\begin{rem}
  Pierre Simon has also observed a version of the implication (2) $\implies$ (1) in 
  Theorem~\ref{thm: characterizing distality} after seeing a preliminary version of our results.
\end{rem}

The above proof shows that an NIP theory is distal if and only if the property (3) in Theorem
\ref{thm: characterizing distality} holds for all finite indiscernible sequences $A,B$. As the
following proposition shows, in an arbitrary NIP theory the property (3) almost holds for $A,B$
indiscernible sequences, except for the uniform definability of one of the homogeneous subsets.

\begin{prop}
  Let $\phi\left(x,y\right)$ be NIP. Then there is $\varepsilon>0$ depending only on $\phi$ such
  that for any $A=\left(a_{i}:i<n\right)$ and $B=\left(b_{i}:i<m\right)$ indiscernible sequences (in
  fact $\Delta$-indiscernible for some finite $\Delta$ depending just on $\phi$ is enough) there are
  $A_{0}\subseteq A,B_{0}\subseteq B$ such that:
  $\left|A_{0}\right|\geq\varepsilon\left|A\right|,\left|B_{0}\right|\geq\varepsilon\left|B\right|$
  and either $\phi\left(a,b\right)$ holds for all $a\in A_{0},b\in B_{0}$ or
  $\neg\phi\left(a,b\right)$ holds for all $a\in A_{0},b\in B_{0}$.\end{prop}
\begin{proof} By Fact~\ref{fac: NIP finite alternation} there is $k$ such that $\phi\left(x,y\right)$
  can't alternate on an indiscernible sequence more than $k$ times.  We divide $B$ into $k+1$
  intervals of almost equal length. Namely, for $i<k+1$ let
  $$B_{i}=\left\{ b_{j}:i\times \frac{m}{k+1}l \leq j<\left(i+1\right)\times \frac{m}{k+1}
  \right\}.$$
  Then for every $a\in A$ there is some interval $B_{i_{a}}$ not containing any alternation
  points. It follows that for some $i'<k+1$, there are $\frac{\left|A\right|}{k+1}$-many points in
  $A$ which do not alternate inside $B_{i'}$, and then at least half of them satisfy $\phi$ or
  $\neg\phi$. So we can take $\varepsilon=\frac{1}{2\left(k+1\right)}$.\end{proof}

\bibliographystyle{amsplain}

\providecommand{\bysame}{\leavevmode\hbox to3em{\hrulefill}\thinspace}
\providecommand{\MR}{\relax\ifhmode\unskip\space\fi MR }
\providecommand{\MRhref}[2]{%
  \href{http://www.ams.org/mathscinet-getitem?mr=#1}{#2} } \providecommand{\href}[2]{#2}

\end{document}